\documentclass[12pt]{amsart}
\usepackage{amsfonts}
\usepackage{amssymb}
\usepackage{amscd}

\usepackage{color}

\input xy
\xyoption{all}
\title[Richardson elements]{Existence of Richardson elements for seaweed 
Lie algebras of type $\mathbb{B}$,  $\mathbb{C}$, $\mathbb{D}$}
\author{Bernt Tore Jensen and Xiuping Su}

\theoremstyle{definition}
\newtheorem{theorem}{Theorem}[section]

\newtheorem{lemma}[theorem]{Lemma}
\newtheorem{proposition}[theorem]{Proposition}
\newtheorem{definition}[theorem]{Definition}
\newtheorem{example}[theorem]{Example}
\newtheorem{question}[theorem]{Question}

\newcommand{\Hom}{\mathrm{Hom}}

\newtheorem{remark}[theorem]{Remark}

\newcommand{\dimv}{\mathrm{\underline {dim}}}

\newcommand{\supp}{\mathrm{supp}}
\renewcommand{\hom}{\mathrm{Hom}}

\newcommand{\ext}{\mathrm{Ext}}

\newcommand{\minus}{\backslash}

\newcommand{\lra}{\longrightarrow}

\newcommand{\ra}{\rightarrow}
\newcommand{\sdp}{\times\kern-.2em\vrule height1.1ex depth-.05ex}
\newcommand{\epi}{\lra \kern-.8em\ra}

\newcommand{\g}{\mathfrak{g}}
\newcommand{\h}{\mathfrak{h}}
\renewcommand{\a}{\mathfrak{a}}
\renewcommand{\c}{\mathfrak{c}}
\renewcommand{\l}{\mathfrak{l}}

\normalbaselines
\thanks{This work was supported by EPSRC 1st grant EP/1022317/1. The results were written up
during the authors' visit to the University of New South Wales. 
Both authors would like to thank J. Du and the School of Mathematics and
Statistics for their hospitality. 
}

\begin{document}

\begin{abstract}  {Seaweed Lie algebras are  a natural generalisation of parabolic subalgebras of reductive Lie algebras. The well-known 
Richardson Theorem says that the adjoint action of a parabolic group has a dense open orbit in the nilpotent radical of its Lie algebra \cite{richardson}.
We call elements in the open orbit 
Richardson elements. 
In \cite{JSY} together with Yu, we generalized  Richardson's Theorem and showed that  Richardson elements exist for seaweed Lie algebras of type $\mathbb{A}$. 
Using GAP, we checked that Richardson elements exist for all exceptional simple Lie algebras except $\mathbb{E}_8$, where we found a counterexample. 

In this paper, we complete the story on Richardson elements  for seaweeds of finite type, by  
showing that  they exist for any seaweed Lie algebra of type $\mathbb{B}$, $\mathbb{C}$ and $\mathbb{D}$. By decomposing a seaweed into a sum of subalgebras and analysing 
their stabilisers, we obtain a sufficient condition for the existence of Richarson elements. The sufficient condition is then verified using quiver representation theory.
More precisely, using the categorical construction of Richardson elements in type $\mathbb{A}$, we prove that the sufficient condition is satisfied for all seaweeds of type $\mathbb{B}$, $\mathbb{C}$ and $\mathbb{D}$, except in two special cases, where we give a direct
proof.}\\


\end{abstract}

\maketitle

\section{Introduction}

Throughout the field $k=\mathbb{C}$,
$\mathfrak{g}$ is a reductive Lie algebra and $G$ is a connected reductive algebraic group with 
Lie algebra $\mathfrak{g}$.

\begin{definition} \cite{S1, P}
A Lie subalgebra $\mathfrak{q}$ of $\mathfrak{g}$ is called a \emph{seaweed subalgebra} 
if there exists a pair $(\mathfrak{p},\mathfrak{p}')$ of parabolic subalgebras of $\mathfrak{g}$
such that $\mathfrak{q}=\mathfrak{p}\cap \mathfrak{p}'$ and $\mathfrak{p}+\mathfrak{p}'=\mathfrak{g}$.
Two such parabolic subalgebras are said to be \emph{opposite}.
\end{definition}

Seaweed Lie algebras (later also called biparabolic algebras, see e.g. \cite{B1, B2}) were defined by Dergachev and Kirillov in their 
study of indexes of Lie algebras \cite{S1} and was generalised 
by Panyushev to arbitrary reductive Lie algebras \cite{P}. 
By definition, parabolics are seaweed Lie algebras. Substantial work on seaweeds
has been done on generalising results on parabolic algebras and beyond. Among others, there are further works on 
indexes by Joseph \cite{B1, B2}, on affine slices for the coadjoint action by Yu and Tauvel \cite{Tauvel2} and Joseph \cite{B2, B3}. 
Also, Panyushev and Yakimova study meander graphs in \cite{PY1, PY2}.
Along the line of generalising results on parabolic algebras, Baur and Moreau 
study quasi-reductive biparabolic algebras in \cite{BM}.

We are interested in  the adjoint action of a seaweed 
Lie algebra on its nilpotent radical and the density of the action, with a view from quiver representation theory. 
\begin{definition}
An element $x$ in the nilpotent radical $\mathfrak{n}$ of a seaweed Lie algebra $\mathfrak{q}$ is called a 
{\it Richardson element}  if $[\mathfrak{q},x]=\mathfrak{n}$.
\end{definition}

In the case where $\mathfrak{q}$ is parabolic, a well-known theorem of Richardson 
\cite{richardson} (see also \cite[Chapter 33]{TYbook}) says that Richardson elements exist. In this case, Br\"{u}stle, 
Hille, Ringel and R\"{o}hrle gave a categorical construction of Richardson elements in type $\mathbb{A}$, using representations of a double 
quiver (with relations) of a linear quiver \cite{BHRR}. 
The following natural question was raised independently by Duflo and Panyushev.

\begin{question}\cite{JSY}\label{question}
Does a seaweed Lie algebra have a Richardson element?
\end{question}

Surprisingly, 
a quiver model can also be constructed from a given seaweed Lie algebra to understand the adjoint action on the nilpotent radical. In \cite{JSY} together with Yu, we made use of the construction in \cite{BHRR} to build rigid modules for the quiver model and obtained a positive answer to Question \ref{question} for 
all seaweeds of type $\mathbb{A}$. An example (Example \ref{newexample}) is given in Section \ref{section3} to illustrate the construction.  This quiver model is the double quiver  (with relations) of a quiver of type $\mathbb{A}$. The path algebra of the double quiver with relations is quasi-hereditary and had been studied by Hille and Vossieck in work on the 
radical bimodule of a hereditary algebra \cite{HV}.

In this paper we prove the existence of Richardson
elements for seaweed Lie algebras of type $\mathbb{B}, \mathbb{C}$ and $\mathbb{D}$. 

\begin{theorem}\label{main0}
Let $\g$ be a Lie algebra of type $\mathbb{B}$,  $\mathbb{C}$ or $\mathbb{D}$. 
Then any seaweed Lie algebra in $\g$ has a Richardson element. 
\end{theorem}

A natural approach to answer Question \ref{question} would be to adapt Richardson's proof for the case of parabolic algebras. However, the fact that the corresponding parabolic group is the normaliser of the nilpotent radical plays an important role in Richardson's proof, but fails for seaweed Lie algebras. At this point,  
we emphasise the advantages of techniques from quiver representation theory, which do not require this fact, as can be seen in 
Br\"{u}stle, Hille, Ringel and R\"{o}hrle's construction for parabolic algebras \cite{BHRR} and our construction for seaweeds 
  \cite{JSY}. 
The key ingredient of the categorical approach is the interplay between Lie algebras and quiver representation theory. For instance, seaweed Lie algebras of type $\mathbb{A}$ are endomorphism algebras of projective representation of a quiver of type $\mathbb{A}$. 
Further, endomorphism algebras of representations that give us Richardson elements in seaweeds of type  $\mathbb{A}$  correspond to stabilisers of the Richardson elements. In this paper we exclusively analyse 
local properties of endomorphisms at a vertex of the quiver and apply these properties to prove the main theorem. 


As a consequence of Theorem \ref{main0} and the results in \cite{JSY}, 
Richardson elements exist for  all seaweed Lie algebras of finite type but $\mathbb{E}_8$.

The remainder of this paper is organised as follows. In Section 2, we recall the notion of a standard
seaweed Lie algebra and prove some important lemmas.  
We decompose standard seaweeds, including standard parabolics, as sums of 
subalgebras and analyse how their stabilizers act.
Further, we show that a local property of stabilisers of Richardson elements for seaweeds of type 
$\mathbb{A}$ is sufficient for the existence of Richardson elements for all seaweeds of other types, 
except in two special cases. 
This condition can be verified using the categorical construction of Richardson elements \cite{BHRR, JS, JSY} 
in type $\mathbb{A}$. 
As such, we recall the  construction  in type 
$\mathbb{A}$ and explain the link between seaweed Lie algebras and representations of quivers in Section 3. We prove essential results on stabilisers in Section 4. 
In Section 5, we prove the main results. In both the general and one of the two special cases, techniques from quiver representation theory
play an important role in the proofs.


\section{Richardson elements and decomposition of seaweeds}

\subsection{Standard seaweeds and parabolics}

We fix a Borel subalgebra $\mathfrak{b}$ of $\mathfrak{g}$ and a
Cartan subalgebra $\mathfrak{h}$ contained in $\mathfrak{b}$.
Denote by $\Phi$, $\Phi^{+}$, $\Phi^{-}$ and $\Pi$, respectively,  the root system,  
the set of positive roots, the set of negative roots and the set of positive simple roots, 
determined by $\mathfrak{h}$, $\mathfrak{b}$ and
$\mathfrak{g}$. For $\alpha \in \Phi$, denote by $\mathfrak{g}_{\alpha}$ the root space 
corresponding to
$\alpha$. Write $$\alpha=\sum_{\alpha_i\in \Pi}x_i\alpha_i.$$  We say that $\alpha$ is supported at a 
positive simple root $\alpha_i$ if $x_i\not= 0$ and call the set of all such simple roots the {\it support} of $\alpha$. 
For $S, T\subset \Pi$, let $\Phi_S$ be the set of roots with support in $S$, 
$$
\Phi_{S}^{\pm}=\Phi^{\pm}\cap \Phi_{S}, \ 
\mathfrak{p}_{S}^{\pm} =
\mathfrak{h} \oplus \bigoplus_{\alpha\in \Phi_{S}^{\mp}\cup \Phi^{\pm}} \mathfrak{g}_{\alpha}  \hbox{ and }
\mathfrak{q}_{S,T}=\mathfrak{p}^{-}_S\cap \mathfrak{p}^+_T.
$$
Note that $\mathfrak{p}_S^{\pm}$ are parabolic subalgebras and $\mathfrak{q}_{S, T}$ is a seaweed 
Lie algebra. Parabolics and seaweeds constructed in this way are said to be \emph{standard} with 
respect to the choice of $\mathfrak{h}$ and $\mathfrak{b}$.

\begin{proposition}{\rm\cite{P,TYbook}} \label{standardseaweed}
Any seaweed Lie algebra in $\mathfrak{g}$ is $G$-conjugate to a standard seaweed Lie algebra.
\end{proposition}

As a consequence, it suffices to consider standard seaweeds when proving the existence of Richardson elements.

Let $\Phi^{+}_{S,T}=\Phi_{S}^{+} \setminus \Phi_{S\cap T}$,
$\Phi^{-}_{S,T}=\Phi_{T}^{-} \setminus \Phi_{S\cap T}$ and
$\Phi_{S,T}=\Phi_{S,T}^{+}\cup \Phi_{S,T}^{-}$.  We have
$$
\mathfrak{q}_{S,T} = \mathfrak{n}_{S,T}^{-} \oplus \mathfrak{l}_{S,T} \oplus \mathfrak{n}_{S,T}^{+}$$
where
$\mathfrak{n}_{S,T}^{\pm} = \bigoplus_{\alpha\in \Phi_{S,T}^{\pm}} \mathfrak{g}_{\alpha}$ 
and $\mathfrak{l}_{S,T} = \mathfrak{h} \oplus \bigoplus_{\alpha\in \Phi_{S\cap T}} \mathfrak{g}_{\alpha}$.
Then $\mathfrak{l}_{S,T}$ is the Levi-subalgebra of $\mathfrak{q}_{S, T}$ and
$\mathfrak{n}_{S,T}=\mathfrak{n}_{S,T}^{+} \oplus \mathfrak{n}_{S,T}^{-}$
is the nilpotent radical of $\mathfrak{q}_{S,T}$.

{\it In the sequel, we will assume that neither $S$ or $T$ is equal to $\emptyset$ or $\Pi$}.
Note that  two algebras of the same type may have different rank and
we view $$\mathbb{B}_1=\mathbb{C}_1=\mathbb{D}_1=\mathbb{A}_1.$$
 By the {\it type of a seaweed} $\mathfrak{q}\subseteq \mathfrak{g}$ we mean the type of $\mathfrak{g}$ and thus 
 the type of $\mathfrak{q}$ is not well-defined without an embedding $\mathfrak{q}\subseteq \mathfrak{g}$.

We note the following symmetry with respect to the choice of $S$ and $T$.

\begin{lemma} \label{symmetry} The seaweed 
$\mathfrak{q}_{S,T}$ has a Richardson element if and only if so does the seaweed $\mathfrak{q}_{T,S}$.
\end{lemma}
\begin{proof}
The lemma follows from the involution $\g\rightarrow \g$ mapping $\g_{\alpha}$ onto $\g_{-\alpha}$.
\end{proof}

\subsection{A decomposition of seaweed subalgebras and Richardson elements }

Let $\mathfrak{q}_{S, T}$ be a standard seaweed in a simple Lie algebra $\g$ of type $\mathbb{B}, \mathbb{C}$ or 
$\mathbb{D}$. Note that $$\g = (\bigoplus_{\alpha\in \Phi}\g_{\alpha})\oplus 
(\bigoplus_{\alpha\in \Pi}[\g_{\alpha},\g_{-\alpha}]),$$ and when $\alpha+\beta\not\in \Phi\cup\{0\}$,
\begin{equation}\label{Eq2} [\g_\alpha,\g_\beta]=0\tag{a} \end{equation}

Denote the positive simple roots of $\g$ by $\alpha_1, \dots, \alpha_n$, with the corresponding Dynkin graph numbered as follows. 

$$\xymatrix{\mathbb{B}: & n\ar@{-}[r] & n-1 \ar@{-}[r] & n-2 & \cdots & 3 \ar@{-}[r]  & 2 \ar@{=>}[r] & 1\;\;}$$

$$\xymatrix{\mathbb{C}: & n\ar@{-}[r] & n-1 \ar@{-}[r] & n-2 & \cdots & 3 \ar@{-}[r]  & 2  & 1\ar@{=>}[l]\;\;}$$

$$\xymatrix{&&&&&&& 1 \\ \mathbb{D}: & n\ar@{-}[r] & n-1 \ar@{-}[r] & n-2 & \cdots & 4 \ar@{-}[r]  & 3 \ar@{-}[ur] \ar@{-}[dr] \\ &&&&&&& 2}$$

{
Let $C=(c_{ij})$ be the Cartan matrix of $\g$, for instance when $\g$ is of type 
$\mathbb{B}_3$, the matrix $C=\left(\begin{matrix}   2&-1&0\\ -1&2&-2\\ 0&-1&2\end{matrix}\right)$. The Cartan matrix of  
type $\mathbb{C}_n$ is the transpose of the Cartan matrix of type $\mathbb{B}_n$.
Let $e_i, f_i, \overline{h}_i$ be the Chevalley generators of 
 $\g$. 
That is, $\g$ is generated by 
the generators subject to the relations. 
 
\begin{itemize}
\item[(1)] $[e_i, f_j]=\delta_{ij} \overline{h}_i$;
\item[(2)] $[ \overline{h}_i, e_j]=c_{ji}e_j$;
\item[(3)] $[ \overline{h}_i, f_j]=-c_{ji}f_j$.
\end{itemize}

We choose a new basis $h_1,\cdots,h_n$ of the Cartan subalgebra of $\mathfrak{g}$ as follows. 
When $\g$ is of type $\mathbb{B}$, $h_1=\frac{1}{2}\overline{h}_1$; when $\g$ is of type $\mathbb{C}$, $h_1= \overline{h}_1$; 
when $\g$ is of type $\mathbb{D}$, $h_1=\frac{1}{2}( \overline{h}_1- \overline{h}_2)$. 
For $i\geq 2$, let
$$h_i= \overline{h}_i+h_{i-1}.$$
}

\begin{lemma}\label{cartannewbasis}
If $i=2$ and $\g$ is of  type $\mathbb{D}$, then 
$$[h_i,\mathfrak{g}_{\pm\alpha_j}]=\left\{\begin{tabular}{ll} $\mathfrak{g}_{\pm\alpha_j}$ & if $j=1, 2, 3,$ \\ $0$ & otherwise.\end{tabular}\right. $$
{For all other types and for all other $i$ when $\g$ is of type $\mathbb{D}$,}
$$[h_i,\mathfrak{g}_{\pm\alpha_j}]=\left\{\begin{tabular}{ll} $\mathfrak{g}_{\pm\alpha_j}$ & if $j=i, i+1,$ \\ $0$ & otherwise.\end{tabular}\right. $$
\end{lemma}

\begin{proof}
{Direct computation gives the following, 
$$
[h_1, e_1]=\left\{\begin{tabular}{ll}
 $2e_1$ & if $\g$ is of type $\mathbb{C}$;\\ $e_1$ &  otherwise. \\ \end{tabular}\right.
$$
For $i>1$, }
$$
[h_i, e_j]=\left\{\begin{tabular}{ll}
 $e_j$ & if $j=i;$\\ $-e_j$ & if $j=i+1;$ \\ $0$ & otherwise, \\ \end{tabular}\right.
$$

except when $i=2$ and $\g$ is of type $\mathbb{D}$, where we have 
$$
[h_2, e_j]=\left\{\begin{tabular}{ll} $e_j$ & if $j=1$ or $2;$\\ $-e_j$ & if $j=3;$ \\ $0$ & otherwise.\\ \end{tabular}\right.
$$
So the lemma follows. 
\end{proof}

Let $\mathfrak{h}_i$  be the subspace spanned by $h_i$. Let
$$\epsilon=\mathrm{min}\{i\mid \alpha_i \not\in S \} \mbox{ and  }  \eta =\mathrm{min}\{i\mid \alpha_i \not\in T \}.
$$
By Lemma \ref{symmetry}, we may assume that $\epsilon\geq \eta\geq 1$. 
Let
$$\omega=\mathrm{max}\{i\mid i\leq \epsilon, \, \alpha_i\not\in T\}. $$
We define two subspaces of $\g$, 
$$\g_1=(\bigoplus_{\alpha\in \Phi_{\{\alpha_i|i<\epsilon\}}}\g_{\alpha})\oplus \bigoplus_{i<\epsilon}\mathfrak{h}_i,$$
and
$$\g_2=(\bigoplus_{\alpha\in \Phi_{\{\alpha_i|i>\omega\}}}\g_{\alpha})\oplus \bigoplus_{i\geq\omega}\mathfrak{h}_i.
$$

\begin{lemma}
\begin{itemize}
\item[(1)] If  $\epsilon>2$ when $\g$ is of type $\mathbb{D}$ or $\epsilon>1$ for other types, 
then $\g_1$ is a Lie subalgebra of the same type as $\g$.
\item[(2)] The subspace $\g_2$ is a Lie subalgebra isomorphic to $\mathfrak{gl}_{n-\omega+1}$.
\end{itemize}
\end{lemma}

\begin{proof}
By the definition of $\mathfrak{h}_i$, we have 
$$\bigoplus_{i<\epsilon}\mathfrak{h}_i=\bigoplus_{i<\epsilon}[\g_{\alpha_i}, \g_{-\alpha_i}],$$
and 
$$ \bigoplus_{i\geq\omega}\mathfrak{h}_i = \bigoplus_{i>\omega}[\g_{\alpha_i}, \g_{-\alpha_i}]\oplus \h_{\omega}.
$$
So the lemma follows. 
\end{proof}

\begin{lemma} \label{specialcase}
If $\epsilon=\omega$, then $\mathfrak{q}_{S,T}$ has a Richardson element.
\end{lemma}
\begin{proof}
If $\omega=\epsilon=1$, then the root spaces in $\mathfrak{q}_{S,T}$ are supported on a   
subgraph of type $\mathbb{A}$, and so $\mathfrak{q}_{S,T}$ is isomorphic to a seaweed
of type $\mathbb{A}$. Similarly, if 
$\omega=\epsilon=2$ and $\mathfrak{g}$ is of type $\mathbb{D}$, then
$\mathfrak{q}_{S,T}$ is also isomorphic to a seaweed of type $\mathbb{A}$.
If $\omega=\epsilon=n$, then $\mathfrak{q}_{S,T}$ is  a reductive Lie algebra 
of the same type as $\mathfrak{g}$ with the nilpotent radical $0$ and thus obviously $0$
is the Richardson element. In the first two cases, by Theorem 1.2 in 
\cite{JSY} or Richardson's Theorem \cite{richardson}, $\mathfrak{q}_{S,T}$ 
has a Richardson element.

In all the other cases, $$\mathfrak{q}_{S,T}=\mathfrak{q}_2\oplus \mathfrak{q}_1,$$ 
where $\mathfrak{q}_1\subseteq \mathfrak{g}_1$ is parabolic of the same type as $\mathfrak{g}$, and
$\mathfrak{q}_2\subseteq \mathfrak{g}_2$ is a seaweed of type $\mathbb{A}$. Further, by Equation (\ref{Eq2}),
$$[\g_1, \g_2]=0 \mbox{ and so }[\mathfrak{q}_1,\mathfrak{q}_2]=0.$$
Since both $\mathfrak{q}_1$ and $\mathfrak{q}_2$ have Richardson elements,  we can conclude that $\mathfrak{q}_{S,T}$
has a Richardson element.
\end{proof}

{\it Consequently we may assume that $\epsilon>\omega$ for the remainder of the paper. 
When $\mathfrak{g}$ is of type $\mathbb{D}$, we
aslo assume  $(\epsilon, \omega)\not= (2, 1)$}, in which case we prove separately the existence of Richardson elements in Theorem \ref{lemmam2}.   
Let
$$
S''=\{\alpha_i\in S\mid i>\omega \}, \;\;T''=\{\alpha_i\in T\mid i>\omega\},$$ 
$$S'= \{\alpha_i\in S\mid i<\epsilon \} \mbox{ and } T'=\{\alpha_i\in T\mid i<\epsilon\}.$$ 
Note that by the defintion of $\epsilon$, $S'$ contains all the simple roots $\alpha_i$ with $i<\epsilon$. These subsets determine two subalgebras of $\mathfrak{q}_{S, T}$, 
 the positive parabolic subalgebra $\c_{S, T}$ of $\g_1$ determined by $T'$ and  the seaweed Lie subalgebra 
$\a_{S, T}$ of $\g_2$ determined by $S''$ and $T''$.

\begin{example}
Let $\g$ be a Lie algebra of type $\mathbb{D}_6$, 
$S=\{\alpha_5, \alpha_3, \alpha_2, \alpha_1\}$ and $T=\{\alpha_6, \alpha_4, \alpha_2, \alpha_1\}$. Then 
$\epsilon=4$, $\omega=\eta=3$. The subalgebras $\a_{S, T}$ and $\c_{S, T}$ can for instance be described using matrices as follows, where $\a_{S, T}$ 
is marked by $\ast$ and $\dagger$, and  $\c_{S, T}$ is marked by $\star$ and $\dagger$. The {one} dimensional intersection is marked by $\dagger$, 
and there is an anti-symmetry to the anti-diagonal.

$$\left(\begin{tabular}{llllllllllll}
$*$&  &  &  &  &  &  &  &  &  & & \\
$*$&$*$&$*$& &  &  &  &  &  &  &  &  \\
 & &$*$& &  &  &  &  &  &  &  &  \\
 & &$*$&$\dagger$& $ \star$ & $ \star$ &$ \star$ & $ \star$ &$ \star$ &  &  & \\
 & & &  &${ \star}$ & $ \star$ &$ \star$ & $ \star$ &$ \star$ &  & & \\
 & & &  &$ \star$ & $ \star$ &$ \star$ & $ \star$ &$ \star$ &  & & \\
 & & & &$ \star$ & $ \star$ &$ \star$ & $ \star$ &$ \star$ &  & & \\
 & & &  &$ \star$ & $ \star$ &$ \star$ & ${ \star}$ &$ \star$ &  & & \\
 &  & & & & & & & $\dagger$ &  & & \\
 &  & & & & & & &$*$&$*$&$*$& \\
 &  & & & & & & & & &$*$& \\
 &  & & & & & & & & &$*$&$*$\\
\end{tabular}\right)
$$
\end{example}

\vspace{3mm}

 Let $\mathfrak{l}=\g_2\cap \g_1$. 
Then
$$\mathfrak{l}=(\bigoplus_{\alpha\in \Phi_{\{\alpha_i|\epsilon>i>\omega\}}} 
\g_\alpha)\oplus \bigoplus_{\epsilon> i\geq\omega} \h_i.
$$ 
Let $\mathfrak{n}_{\a}$ and $\mathfrak{n}_{\c}$
be the nilpotent radicals of $\a_{S,T}$ and $\c_{S,T}$, respectively.

\begin{lemma}\label{decomposition}
Assume $(\epsilon,\omega)\neq (2,1)$ if $\mathfrak{g}$ is of type $\mathbb{D}$. Then
\begin{itemize}
\item[(1)] $\mathfrak{q}_{S,T}=\a_{S,T}+\c_{S,T}$ 
\item[(2)] $\a_{S,T}\cap \c_{S,T}=\mathfrak{l}$ is a block in the Levi subalgebra of $\mathfrak{q}_{S, T}$. 
\item[(3)] $\mathfrak{n}_{S, T}=\mathfrak{n}_{\a}\oplus \mathfrak{n}_{\c}$
\end{itemize}
\end{lemma}
\begin{proof}
(1) Let $\alpha$ be a positive root such that $\g_\alpha\subseteq \mathfrak{q}_{S,T}$. Then $\alpha$ is not supported at
$\alpha_\epsilon$, since $\epsilon\not\in S$. By the assumption on the root system,  $\alpha$ must be supported on simple roots
$\alpha_i$ with all $i<\epsilon$ or all $i>\epsilon$. So $\g_\alpha\subseteq \a_{S,T}$ or $\g_{\alpha}\subseteq \c_{S,T}$.
Similarly, a negative root $\beta$ with $\g_{\beta}\subseteq \mathfrak{q}_{S,T}$ is supported on 
simple roots $-\alpha_i$ with all $i<\omega$ or all $i>\omega$, and so $\g_\beta\subseteq \a_{S,T}$ or $\g_{\beta}\subseteq \c_{S,T}$.
By construction,  $[\g_{\alpha_i},\g_{-\alpha_i}]\subseteq \a_{S,T} + \c_{S,T}$ for all simple roots $\alpha_i$, and so $\mathfrak{q}_{S,T}=\a_{S,T}+\c_{S,T}$.

(2) follows from the construction and 
(3) follows from (1) and (2).
\end{proof}

By Theorem 1.2 in \cite{JSY},  Richardson elements exist in 
$\a_{S, T}$.
Let $r_2\in \a_{S, T}$ be a Richardson element and denote by $$\mathrm{stab}_{\mathfrak{a}_{S,T}}(r_2)=
\{x\in \mathfrak{a}_{S,T}|[x,r_2]=0\}, $$ the stabiliser of $r_2$ in 
$\mathfrak{a}_{S,T}$.
For a subalgebra $\mathfrak{u}\subseteq \g$ given 
as a direct sum of root spaces and subspaces  $\mathfrak{h}_i$ 
let $x_{|\mathfrak{u}}$ be the canonical projection of $x\in \g$ onto $\mathfrak{u}$. 
Let $$\c_{r_2}=
\{x\in \c_{S,T} | x_{|\mathfrak{l}} = y_{|\mathfrak{l}} \mbox{ for some } y\in \mathrm{stab}_{\a_{S,T}}(r_2)\}.$$

\begin{lemma} \label{lemma2} Assume $(\epsilon,\omega)\neq (2,1)$ if $\mathfrak{g}$ is of type $\mathbb{D}$ and let $r_1\in \mathfrak{n}_{\mathfrak{c}}$.
If $[\c_{r_2},r_1]=\mathfrak{n}_{\mathfrak{c}}$, then $r_1+r_2$ is a Richardson element of the seaweed $\mathfrak{q}_{S,T}$.
\end{lemma}
\begin{proof}
Assume $[\c_{r_2},r_1]=\mathfrak{n}_{\mathfrak{c}}$. Take any $(x_{\a},x_{\c})\in 
\mathfrak{n}_{\a}\oplus \mathfrak{n}_{\c}$. There exists $y_{\a}\in \a_{S,T}$ 
such that $$[y_\a,r_2]=x_\a.$$ Write $y_{\a}=y'_{\a}+(y_{\a})_{|\l}$. 
Note that $r_1|_{\g_{\alpha}}\not= 0$ can occur only for  positive roots $\alpha$ with support contained in $\{\alpha_{\epsilon-1}, \dots, \alpha_1\}$,
and $y_{\mathfrak{a}}'|_{\mathfrak{g}_{\beta}}\not= 0$ can occur only for positive roots $\beta$ with support contained in $\{\alpha_n, \dots, 
\alpha_{\epsilon+1}\}$, 
or negative roots with support contained in $\{\alpha_n, \dots, \alpha_{\omega+1}\}$ and containing at least one $\alpha_j$ for some $j\geq \epsilon$. 
So  by Equation (\ref{Eq2}) in Section 2.2 and the fact from  that 
$$[\h_i, g_{\alpha_j}]=0 \text{ for } i\geq \epsilon \text{ and } j<\epsilon$$ by
Lemma \ref{cartannewbasis}, we have
\begin{equation}\label{Eq1}[y'_{\a},r_1]=0\tag{b}\end{equation}
and so 
$$[y_{\a},r_1]=[(y_{\a})_{|\l},r_1]\in \mathfrak{n}_{\mathfrak{c}}.$$

Let $y_\c\in \c_{r_2}$ be such that
$$[y_\c,r_1]=x_\c-[y_\a,r_1].$$
Let $z\in \mathrm{stab}_{\a_{S,T}}(r_2)$ with $z_{|\mathfrak{l}}=(y_{\c})_{|\mathfrak{l}}$ and $z'=z-z_{|\mathfrak{l}}$.
Then similar to Equation (\ref{Eq1}), 
$$[z',r_1]=0$$
and $$[y_\c-{(y_\c)}_{|\mathfrak{l}},\; r_2]=0.$$
Therefore
$$\aligned
&\; [y_\a+z'+y_\c,r_1+r_2 ] \\
=&\; [y_\a,r_1+r_2]+[z',r_1+r_2]+[y_\c, r_1+r_2]\\
=&\; x_\a+[y_\a,r_1] + [z',r_2] + [z_{|\mathfrak{l}},r_2]+x_\c-[y_\a,r_1]\\
=& \;x_\a+x_\c+[z,r_2]\\
=&\; x_\a+x_\c.
\endaligned$$
This completes the proof of the lemma.
\end{proof}

\subsection{A decomposition of parabolic subalgebras and Richardson elements}\label{generalassumption}   
The main goal in this subsection is to present a key sufficient condition for the existence of Richardson elements in general, except the two special cases, i.e. when $\epsilon=\omega$ as in Lemma \ref{specialcase} and when 
$(\epsilon,\omega)= (2,1)$ in type $\mathbb{D}$ . We transfer the existence to a local problem between a parabolic subalgebra constructed from a given seaweed Lie algebra and 
the seaweed Lie algebra itself. 
We first give a decomposition of parabolic subalgebras, discuss properties of subalgebras in the decomposition and then state and prove the sufficient condition at the end of the section.

Let $S, \,T,\,\epsilon$, $\omega$, $\mathfrak{g}_1$, $\mathfrak{g}_2$ and $\mathfrak{l}$ be defined as in Section 2.2 with $\omega<\epsilon$. 
{When $\g$ is of type $\mathbb{D}$, we  continue to assume that  $$ (i) \;\; (\epsilon,\omega)\neq (2,1).$$ Further, we assume that $$(ii)\;\; \eta\not=2 \text{ when } \epsilon>2.$$ The assumption $(ii)$ is purely a technical issue, to avoid a complication in the description of the decomposition 
discussed in this subsection and it does not compromise the completeness of the existence of Richardson elements 
for $\mathfrak{q}_{S, T}$ with $(\epsilon, \omega)\not=(2, 1)$, due to the symmetry between $\alpha_1$ and $\alpha_2$ when $\g$ is 
of type $\mathbb{D}$. {\it In the remaining of this section we assume  $\mathfrak{q}_{S, T}$ satisfy both (i) and (ii)}. }

Let $\mathfrak{g}'$ be a Lie algebra of the same type as $\mathfrak{g}$, with 
rank at least $\epsilon$ and root system denoted by $\Phi'$. We may assume that 
both $\mathfrak{g}$ and $\mathfrak{g'}$
are subalgebras of a Lie algebra of the same type as $\mathfrak{g}$ such that
$\mathfrak{g}\subseteq \mathfrak{g'}$ or $\mathfrak{g'}\subseteq \mathfrak{g}$. 
Here all inclusions are induced by inclusions of Dynkin diagrams.

Let $\mathfrak{p}^+_U\subseteq \g'$ be the standard parabolic subalgebra determined by $U$
with $\epsilon\not\in U$ and $$\{\alpha_i|i<\epsilon\}\cap U=\{\alpha_i|i<\epsilon\}\cap T.$$ 
We choose a basis $\{\h'_i\}_i$ for the Cartan subalgebra $\mathfrak{h}'$ of $\mathfrak{g}'$ 
in the same manner as we did for the basis $\{\h_i\}_i$ of the Cartan subalgebra of $\g$.
Let $\g'_1=\g_1$ and let $\g'_2\subseteq \g'$ be defined
similarly to $\g_2\subseteq \g$, i.e. 
$$(\bigoplus_{\alpha\in \Phi_{\{\alpha_i|i>\omega\}}}\g'_{\alpha})\oplus \bigoplus_{i\geq\omega}\mathfrak{h}'_i,$$
which is of type $\mathbb{A}$.
Further, let  $U''=\{\alpha_i\in U|i>\omega\}$, $U'=\{\alpha_i\in U|i<\epsilon\}.$
These two sets determine the following standard parabolic subalgebras of $\g_2'$ and $\g_1'$,
$$\a_U = \bigoplus_{\alpha\in \Phi'^-_{U''}\cup \Phi'^+_{\{\alpha_i|i>\omega\}}} \g'_{\alpha} \oplus \bigoplus_{i\geq\omega} \h'_i\subseteq \g'_2$$ 
and $$\c_U = \bigoplus_{\alpha\in \Phi^-_{U'}\cup \Phi^+_{\{\alpha_i|i<\epsilon\}}} \g'_{\alpha} \oplus \bigoplus_{i<\epsilon}\h'_i\subseteq \g'_1.$$
Note that  $\c_{U}=\c_{S,T}$.

Let $\mathfrak{d}_U\subseteq \mathfrak{p}^+_U$ be the direct sum of all root spaces $\g_\alpha$ with $\alpha$ a positive
root such that $\g_\alpha$ is neither contained in $\a_{U}$ nor in $\c_{U}$.  Let $\mathfrak{n}'_\a$ be the nilpotent radical of $\a_U$. 
Recall that $\mathfrak{n}_\c$ is the nilpotent radical of $\c_U=\c_{S,T}$. 

{
\begin{example}\begin{itemize}\item[(1)]
Let $\g$ be a Lie algebra of type $\mathbb{D}_6$, $S=\{\alpha_6, \alpha_4, \alpha_3, \alpha_2, \alpha_1\}$ and 
$T=\{\alpha_6, \alpha_5, \alpha_4, \alpha_2\}$. 
Then $\epsilon=5$, $\omega=3, \eta=1$. The subalgebras $\a_{S, T}$ and $\c_{S, T}$ are as below, marked by $\ast, \dagger$ and 
$\dagger, \star$, respectively, where the intersection is marked by $\dagger$.
$$ \mathfrak{q}_{S, T}=
\left(\begin{tabular}{llllllllllll} 
$*$&$*$  &  &  &  &  &  &  &  &  & & \\
$*$&$*$&$$& &  &  &  &  &  &  &  &  \\
$*$& $*$&$\dagger$& $\dagger$ & $\star$ & $\star$ &$\star$  & $\star$ & $\star$ &$\star$  &  &  \\
$*$ &$*$ &$\dagger$&$\dagger$& $\star$& $\star$ &$\star$  & $\star$ & $\star$ &$\star$  &  &  \\
 & & &  &$\star$ &  $\star$&$ \star$ & $ \star$ & $\star$ &$\star$  & & \\
 & & &  &$\star$ & $\star$ &$ \star$ & $ \star$ &$ \star$ &$\star$  & & \\
 & & & & &   &$ \star$ & $ \star$ &$ \star$ &$\star$  & & \\
 & & &  &  &   &$ \star$ & ${ \star}$ &$ \star$ &$\star$  & & \\
 &  & & & & & & & $\dagger$ &  $\dagger$& & \\
 &  & & & & & & &$\dagger$&$\dagger$&& \\
 &  & & & & & & &$*$ &$*$ &$*$& $*$\\
 &  & & & & & & & $*$& $*$&$*$&$*$\\
\end{tabular}\right)
$$
\item[(2)] Let $\g'=\g$ and $U=\{\alpha_6, \alpha_4, \alpha_2\}$, which satisfies the conditions 
$$\alpha_{\epsilon}\not\in U \;\text{ and }\; U\cap \{\alpha_i\mid i<\epsilon \}=T\cap \{\alpha_i\mid i<\epsilon \}=\{ \alpha_4, \alpha_2\}.$$
The subalgebras $\mathfrak{a}_U$ marked by $\ast$ and $\dagger$,  $\mathfrak{c}_U$ by $\dagger$ and $\star$, and 
$\mathfrak{d}_U$ by $-$ are as below.

$$
\mathfrak{p}_{U}=
\left(\begin{tabular}{llllllllllll} 
$*$&$*$  & $*$ &$*$  & $-$ & $-$   & $-$   &  $-$  & $-$   &  $-$  & $-$  &  $-$ \\
$*$&$*$&$*$& $*$&  $-$ & $-$   & $-$   &  $-$  & $-$   &  $-$  & $-$  &  $-$   \\
&  &$\dagger$& $\dagger$ & $\star$ & $\star$ &$\star$  & $\star$ & $\star$ &$\star$  & $-$  &  $-$ \\
&  &$\dagger$&$\dagger$& $\star$& $\star$ &$\star$  & $\star$ & $\star$ &$\star$  & $-$  &  $-$ \\
 & & &  &$\star$ &  $\star$&$ \star$ & $ \star$ & $\star$ &$\star$  & $-$ &  $-$\\
 & & &  &$\star$ & $\star$ &$ \star$ & $ \star$ &$ \star$ &$\star$  & $-$ &  $-$\\
 & & & & &   &$ \star$ & $ \star$ &$ \star$ &$\star$  & $-$ & $-$ \\
 & & &  &  &   &$ \star$ & ${ \star}$ &$ \star$ &$\star$  &  $-$& $-$ \\
 &  & & & & & & & $\dagger$ &  $\dagger$& $*$&$*$ \\
 &  & & & & & & &$\dagger$&$\dagger$&$*$& $*$\\
 &  & & & & & & &  &  &$*$& $*$\\
 &  & & & & & & &  &  &$*$&$*$\\
\end{tabular}\right),
$$
\item[(3)] $\mathfrak{a}_{S, T}\cap \mathfrak{c}_{S, T}=\mathfrak{a}_{U}\cap \mathfrak{c}_{U}$, marked by $\dagger$, and 
$\mathfrak{c}_{S, T}=\mathfrak{c}_U$. 
Note that in both (1) and (2) there is an anti-symmetry to the anti-diagonal.
\end{itemize}
\end{example}}

\begin{lemma} \label{lemmabove} The following are true. 
\begin{itemize} 
\item[(1)] $[\mathfrak{p}^+_U,\mathfrak{d}_U]\subseteq \mathfrak{d}_U$.
\item[(2)]  $[\mathfrak{p}^+_U,\mathfrak{n}'_{\mathfrak{a}}]\subseteq \mathfrak{n}'_{\mathfrak{a}}+\mathfrak{d}_U$.
\item[(3)] $\mathfrak{p}^+_U=(\a_U+\c_U)\oplus \mathfrak{d}_U$.
\item[(4)] $\a_U\cap \c_U=\mathfrak{l}$.
\item[(5)] $\mathfrak{n}_U=\mathfrak{n}'_{\mathfrak{a}}\oplus  \mathfrak{n}_{\mathfrak{c}} \oplus \mathfrak{d}_{U}$.
\end{itemize}
\end{lemma}
\begin{proof}
By the construction, 
$\mathfrak{d}_U$ is the direct sum of the root spaces $\g_\alpha$ with $\alpha$ positive and 
supported at both simple roots $\alpha_\epsilon$ and $\alpha_\omega$.  For any $\g_{-\alpha}\subseteq \mathfrak{p}_U^+$ with $-\alpha$
a negative root, the root $-\alpha$ is not supported at $\alpha_{\epsilon}$ and $\alpha_{\omega}$. So (1) follows. Similar, (2) holds.

(3) and (4) follow from the construction. 
(5) follows from (3) and (4).
\end{proof}

By Richardson's theorem, there exists $$r=r_1+r'_2+r_d$$ with $(r_1,r'_2,r_d)\in \mathfrak{n}'_{\mathfrak{a}}
\oplus  \mathfrak{n}_{\mathfrak{c}} \oplus \mathfrak{d}_{U}$ such that $[\mathfrak{p}^+_U,r]=\mathfrak{n}_U$. 
By Lemma \ref{lemmabove} (1) (2), we may assume that $r_1$ is 
the Richardson element for $\mathfrak{c}_{S, T}$ from Section 2.2. 
Again by Lemma \ref{lemmabove}, we can identify
$$(\mathfrak{a}_U+\mathfrak{c}_U)=\mathfrak{p}^+_U / \mathfrak{d}_U
\mbox{ and }\mathfrak{n}'_{\mathfrak{a}}\oplus  \mathfrak{n}_{\mathfrak{c}}=\mathfrak{n}_U/\mathfrak{d}_{U}.$$
So we have a well-defined action $\mathfrak{p}^+_U$ on $\mathfrak{n}'_{\mathfrak{a}}\oplus  \mathfrak{n}_{\mathfrak{c}}$
and $$\mathfrak{n}'_{\mathfrak{a}}\oplus  \mathfrak{n}_{\mathfrak{c}}=
[\mathfrak{p}^+_U, r_1+r'_2]=[\a_U+\c_U,r_1+r'_2]$$
Let $\c_{r'_2}=
\{x\in \c_{U} | x_{|\mathfrak{l}} = y_{|\mathfrak{l}} \mbox{ for some } y\in \mathrm{stab}_{\a_U}(r'_2)\}.$

\begin{lemma} \label{lemma11} We have 
$[\c_{r'_2}, r_1]=\mathfrak{n}_{\c}$.
\end{lemma}
\begin{proof}
Let $x\in \mathfrak{n}_\c$. There exists $y\in \a_U+\c_U$ such that $[y,r_1+r'_2]=x$.
We write $$y=y''+y_{|\mathfrak{l}}+y',$$ 
where $$y''+y_{|\mathfrak{l}}\in \a_U\mbox{ and } y_{|\mathfrak{l}}+y'\in \c_U.$$ Then  similar to Equation (\ref{Eq1}) in the proof 
of Lemma \ref{lemma2},  
$$[y'',r_1]=0, \;[y',r'_2]=0$$
and so
$$[y,r'_2]=[y''+y_{|\mathfrak{l}},r'_2]\in \mathfrak{n}'_\a \mbox{ and } [y,r_1]=[y_{|\mathfrak{l}}+y',r_1]\in \mathfrak{n}_\c.$$ Since 
$\mathfrak{n}'_\a\cap \mathfrak{n}_\c=0$ and $[y,r'_2+r_1]=x\in \mathfrak{n}_\c$, we have 
$$[y,r'_2]=[y''+y_{|\mathfrak{l}},r'_2]=0 \mbox{ and }
[y,r_1]=[y_{|\mathfrak{l}}+y',r_1]=x.$$
It follows that $y_{|\mathfrak{l}}+y'\in \c_{r'_2}$ and so 
$[\c_{r'_2}, r_1]=\mathfrak{n}_{\c}$.
\end{proof}

Recall that $r_2$ is a Richardson element for $\a_{S,T}$.
{\color{black}We have the following key observation, which gives a sufficient condition for the existence of Richardson elements}. 

\begin{lemma} \label{keylemma}
If $\mathrm{stab}_{\a_U}(r'_2)_{|\mathfrak{l}}=\mathrm{stab}_{\a_{S,T}}(r_2)_{|\mathfrak{l}}$,
then $\mathfrak{q}_{S,T}$ has a Richardson element.
\end{lemma}
\begin{proof}
Assume  $\mathrm{stab}_{\a_U}(r'_2)_{|\mathfrak{l}}=\mathrm{stab}_{\a_{S,T}}(r_2)_{|\mathfrak{l}}$.
Then $\c_{r'_2}=\c_{r_2}$ and so $[\c_{r_2},r_1]=\mathfrak{n}_{\c}$ by Lemma \ref{lemma11}.
Then $\mathfrak{q}_{S,T}$ has a Richardson element, by Lemma \ref{lemma2}.
\end{proof}

{\color{black} Verifying the condition in Lemma \ref{keylemma} is a key step in the proof of  the existence of  Richardson elements in the seaweed $\mathfrak{q}_{S,T}$. 
 We will make use of the categorical construction of Richardson elements \cite{JSY}. That is, we will analyze the properties of local endomorphisms (i.e. restrictions of endomorphisms to a vertex) of rigid modules constructed in \cite{JSY}. So we recall the construction in next section.

\section{Rigid $D$-modules and Richardson elements in type $A$}\label{section3}
In this section, we first recall a quasi-hereditary algebra $D$, which is the path algebras of a double quiver (with relations) of a quiver $Q$ of type $\mathbb{A}$,  and the construction of rigid good $D$-modules \cite{JSY}. We then explain in examples how to construct Richardson elements for the corresponding seaweed Lie algebras from the rigid modules. We remark that the type $\mathbb{A}$ quiver $Q$ can constructed from 
a given seaweed and  the relations defining the algebra $D$ can be read off from the seaweed as well \cite{JSY}. }

\subsection{The path algebra $D$ of a double quiver with relations}\label{introducingD}
Let $Q$ be a quiver of type 
$\mathbb{A}_m$ with vertices $Q_0=\{1,\cdots, m\}$  and arrows 
$$Q_1=\{\alpha_i\mid i\rightarrow i+1 \mbox{ or } i\leftarrow i+1 \mbox{ for } i=1,\cdots, m-1\}.$$
Let  $A=kQ$, the path algebra of $Q$. 
We denote the projective indecomposable $A$-module associated to vertex $i$ by $P_i$. Let 
$$P(d)=\bigoplus_{i=1}^mP_i^{d_i}$$ for any $d\in \mathbb{Z}_{\geq 0}^m$.
Note that $\mathrm{End}_AP(d)$ is a seaweed in a Lie algebra of type $\mathbb{A}$, and a Richardson elements in $\mathrm{rad End}_AP(d)$ can be 
constructed from a good rigid representation $X(d)$  \cite{JS,JSY} of a double quiver of $Q$ with relations. 
We recall the double quiver with relations from \cite{HB, HV} and the construction 
of $X(d)$. 

Let $\tilde{Q}$ be the double quiver of $Q$, i.e. $\tilde{Q}_0=Q_0$ and 
$\tilde{Q}_1=Q_1\cup Q_1^*$ with $$Q_1^*=\{\alpha^*: i\rightarrow j| 
\alpha: j\rightarrow i\in Q_1\}.$$ Let $\mathcal{I}$ be the ideal of $k\tilde{Q}$ 
generated by
$$\alpha^*\alpha - \sum_{\beta\in Q_1, t(\beta)=s(\alpha)} \beta\beta^*$$ 
for any arrow $\alpha\in Q_1$, where $s(\alpha)$ is the starting vertex of $\alpha$ and $t(\beta)$ is the terminating vertex of $\beta$; and $$\alpha^*\beta$$ for pairs of arrows
$\alpha\neq \beta$ in $Q_1$ terminating at the same vertex. Let $$D=k\tilde{Q}/\mathcal{I}.$$
Any $D$-module is an $A$-module via the inclusion $A\subseteq D$ and 
any $A$-module is a $D$-module via the surjection $D\twoheadrightarrow A$ mapping all arrows in 
$Q_1^*$ to zero. { We use the notation $_AX$ to indicate the $A$-module structure of a $D$-module $X$
and note that $$\Hom_A(M,N)=\Hom_D(M,N)$$ for two $A$-modules $M$ and $N$.}

The algebra $D$ is quasi-hereditary with Verma modules $P_1, \dots, P_m$ (see \cite{BHRR}). 
The modules filtered by the  Verma modules are called \emph{good modules}. 
So for any good  $D$-module $M$, we have $$_AM\cong P(d)$$ as $A$-modules
for some $d\in \mathbb{Z}^m_{\geq 0}$. We call $d$ 
the {\it $\Delta$-dimension vector} of $M$, denoted by $\dimv_{\Delta}M$, and the set 
$\mathrm{supp}_{\Delta}(M)=\{i\mid d_i\not=0\}$ the {\it $\Delta$-support} of $M$.
This definition is similar to the support of a module, which is defined using the usual dimension vector. 

We identify modules with the 
corresponding quiver representations. So a $D$-module $M$ is a collection of vector spaces $M_i, i\in Q_0$ and linear maps $M_\beta, \beta\in Q_1\cup Q^*_1$, satisfying the relations
$\mathcal{I}$,  and
a homomorphism $f:M\rightarrow N$ of $D$-modules is 
a collection of linear maps $(f_i)_{i\in Q_0}$ commuting with the module structure on $M$ and $N$.

Note that a $D$-module $M$ is \emph{rigid} if it has no self-extensions, i.e. $$\ext^1(M, M)=0.$$ In the remaining of this section, we 
briefly recall the construction of rigid $D$-modules and their corresponding Richardson elements \cite{JSY}. 

\subsection{Construction of rigid $D$-modules: the linear case \cite{BHRR}} 
Let $Q$ be a linear quiver with $m$ the unique sink vertex. Then $\mathcal{I}$ is generated by commutative 
relations at $2, \dots, m-1$, and a zero relation at $1$. In this case
the indecomposable projective $D$-module $R_m$, at vertex $m$, is injective.
A submodule $X$ of $R_m$ 
is uniquely determined by its $A$-structure $_{A}X\cong
\oplus^m_{i=1}P_i^{d_i}$ with $d_i \in \{0,1\}$. Thus there is 
a natural bijection between subsets $I\subseteq Q_0$ and submodules of 
$R_m$. More precisely, under this bijection a subset $I$ corresponds to the 
unique submodule $X(I)\subseteq R_m$ with $\Delta$-support $I$. 
For any vector $d\in \mathbb{Z}_{\geq 0}^m$, define $$X(d)=\sum^t_{i=1} X(I_i),$$
with $\dimv_{\Delta} X(d)=d$ 
and $I_1\subseteq I_2 \subseteq \cdots \subseteq I_t$. Then $X(d)$ is a rigid $D$-modules. 
We give an example to illustrate the construction. See \cite{BHRR} for more details.

\begin{example}\label{example2}
Let $m=3$ and $d=(2,1,2)$. The algebra $D$ is given by the quiver 
$$\xymatrix@=5mm{1 \ar@/^/[r]^{\alpha_1}
&2\ar@/^/[r]^{\alpha_2} \ar@/^/[l]^{\alpha_1^*}
& 3 \ar@/^/[l]^{\alpha_2^*} },$$ with
the ideal $\mathcal{I}$ generated by  $\alpha_1^*\alpha_1$ and
$\alpha_1\alpha_1^*-\alpha_2^*\alpha_2$.
The projective-injective D-module $R_3$ has
the following seven nonzero submodules with  the first one $R_3$,

$$
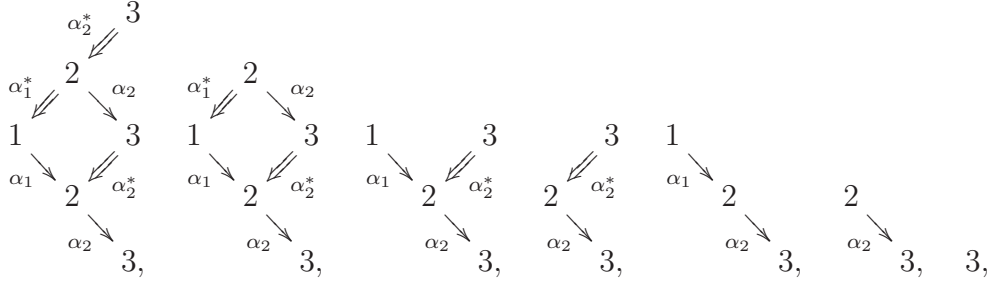
\xymatrix@=3mm{ && 3\ar@{=>}[dl]_{\alpha_2^*} \\
& 2\ar@{=>}[dl]_{\alpha_1^*}\ar[dr]^{\alpha_2} & & & 2\ar@{=>}[dl]_{\alpha_1^*}\ar[dr]^{\alpha_2} \\ 1\ar[dr]_{\alpha_1}
&& 3\ar@{=>}[dl]^{\alpha_2^*} & 1\ar[dr]_{\alpha_1}  
&& 3\ar@{=>}[dl]^{\alpha_2^*} & 1\ar[dr]_{\alpha_1}  
&& 3\ar@{=>}[dl]^{\alpha_2^*}
&& 3\ar@{=>}[dl]^{\alpha_2^*} & 1\ar[dr]_{\alpha_1}\\
& 2\ar[dr]_{\alpha_2} & && 2\ar[dr]_{\alpha_2}  &&&2\ar[dr]_{\alpha_2}  
&&2\ar[dr]_{\alpha_2} &&& 2\ar[dr]_{\alpha_2}  &&
2\ar[dr]_{\alpha_2} \\ && 3,&&&3,&&&3,&&3,&&&3,&&3, & 3,
}$$
corresponding to the subsets $\{1, 2,3\}$, $\{1, 2\}$, $\{1, 3\}$, $\{2, 3\}$, 
$\{1\}$, $\{2\} $, $\{3\}$, respectively. In the picture a number $i$ indicates
a one dimensional basis element at vertex $i$
and the arrows indicate the nonzero action of the arrows in $\tilde{Q}_{1}$. 
We have $$X(d)=X(\{1, 2, 3\})\oplus X(\{1, 3\}).$$ 
\end{example}

\subsection{Construction of rigid $D$-modules: the general case \cite{JS, JSY}}
Now suppose that $Q$ has an arbitrary orientation. Recall that a vertex is admissible if it is a source or a sink. 
Let $$i_1<i_2<\cdots <i_{t-1}<i_t$$ be the complete list of interior admissible vertices 
in $Q$ and let $i_0=1$ and $i_{t+1}=m$.  
Each interval $[i_s, \dots, i_{s+1}]$ has a unique sink and a unique source. 
Similar to the linear case, 
each subset $I \subseteq \{i_s, \dots, i_{s+1}\}$ determines a unique (up to isomorphism)
indecomposable rigid good $D$-module, 
which has $\Delta$-support $I$ and is a submodule of the indecomposable projective module at the 
sink in this interval. 

Two indecomposable rigid good $D$-modules $M$ and $N$ with 
$$\supp_{\Delta}(M)\cap \supp_{\Delta}(N)=\{i_j\},$$ $\supp_{\Delta}(M)\subseteq \{i|i\leq i_j\}$ and
$\supp_{\Delta}(N)\subseteq \{i|i\geq i_j\}$, can be glued by identifying $P_{i_j}$ 
to obtain a new indecomposable rigid good $D$-module.

\begin{definition}\cite{JS}\label{order}
 Let $u$ be a vertex with $i_{v}<u\leq i_{v+1}$. 
Suppose that two indecomposable rigid $D$-modules $M$ and $N$, glued from $X(I_r)$s and $X(J_r)$s, respectively,  are supported 
(but not necessarily $\Delta$-supported) 
at $u$. We define $M\leq_u N$ if for any $s$ with both $I_s$ and $ J_s$ nonempty, 
$I_s\subseteq J_s$ when $s-v$ is even and $I_s\supseteq J_s$ when $s-v$ is odd. 
\end{definition}

{\color{black}
\begin{remark}
The order $\leq_u$ depends on the base interval, i.e. the $v$th interval $[i_v, i_{v+1}]$. When $u=i_{v+1}$ is admissible, $u$ is contained in both the $(v+1)$th interval $[i_{v+1}, i_{v+2}]$ and the $v$th one $[i_v, i_{v+1}]$. We can also define $\leq_u$ based on the $(v+1)$th interval and obtain an order that is opposite to the one defined in Definition \ref{order}. This explains  for instance in Example \ref{example3}, why the summands $M^1$ vs $M^2$ and $N^1$ vs $N^2$ are ordered the way they are. 
\end{remark}}

Now using the order $\leq_u$ we can construct rigid good $D$-modules as follows. Let $M$ and $N$ be two good rigid $D$-modules 
with $\mathrm{dim}_\Delta(M)_i=0$ for $i>i_j$,  $\mathrm{dim}_\Delta(N)_i=0$ for $i<i_j$, 
and $\mathrm{dim}_\Delta(N)_{i_j}=\mathrm{dim}_\Delta(M)_{i_j}$. 
With respect to $\leq _{i_j}$, we glue the $i$th biggest  summand of $M$  to the $i$th biggest summand of $N$. 
In this way, we obtain a rigid good $D$-module $X(d)$ for any $\Delta$-dimension vector $d$.
We illustrate the construction by an 
example. See \cite{JS,JSY} for more details.

\begin{example}\label{example3}
Let $Q$ be the quiver 
$$\xymatrix@=5mm{ 1& 2\ar[l]&3\ar[l]\ar[r]&4\ar[r]&5.}$$
Let $d=( 2, 1, 2, 1, 2)$. Let $d^1=( 2, 1, 2, 0, 0)$ 
and  $d^2=( 0, 0, 2, 1, 2)$. Then $X(d^1)=M^1\oplus M^2$ and 
$X(d^2)=N^1\oplus N^2$ with 
$
M^1=X(\{1, 3\}), M^2=X(\{1,2, 3\}), N^1=X(\{3, 4,  5\})$ and $N^2=X(\{3,5\})$
 as follows, 
$$
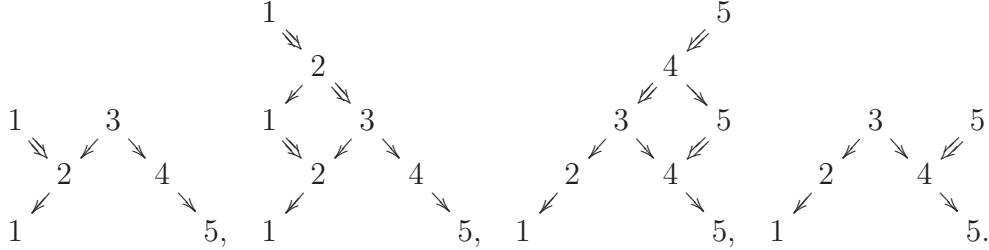
\xymatrix@=2mm{  
&&&& & 1\ar@{=>}[dr] & &&&  &&  & &&   5\ar@{=>}[dl]&&&&& \\
& &&& &&2\ar[dl]\ar@{=>}[dr]    & &&  &&&&4\ar[dr]\ar@{=>}[dl]& &&&&\\
1\ar@{=>}[dr]&&3\ar[dl]\ar[dr]&&& 1\ar@{=>}[dr] && 3\ar[dr]\ar[dl] &&&& 
&3\ar[dl]\ar[dr] && 5\ar@{=>}[dl] & &&  3\ar[dl]\ar[dr]&&5\ar@{=>}[dl] \\
& 2\ar[dl]&&4\ar[dr]&&&2 \ar[dl]&& 4 \ar[dr]
&&& 2\ar[dl] && 4\ar[dr] && & 2\ar[dl]&&4\ar[dr] \\
1 &&&& 5, &1&&&&5,& 1&&&&5, &1 &&&&5. }$$
We have $M^1\leq_{3} M^2$ and  $N^1\leq_3 N^2$. 
So $X(d)$ is the direct sum of the gluings of $M^1$, $M^2$ with $N^1$ and $N^2$, 
respectively, i.e.,  
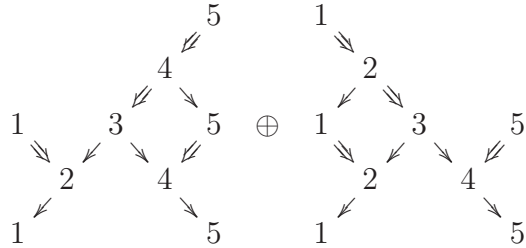
$$\xymatrix@=2mm{
&&&&5\ar@{=>}[dl]& &1\ar@{=>}[dr]   &&& &\\
&  &&  4\ar@{=>}[dl]\ar[dr] &&& &2\ar[dl]\ar@{=>}[dr] & &\\
1\ar@{=>}[dr]&&3\ar[dl]\ar[dr]&& 5\ar@{=>}[dl]& \oplus &
1\ar@{=>}[dr]&&3\ar[dr]\ar[dl]&&5\ar@{=>}[dl]\\
&2\ar[dl] &&4\ar[dr]    &&&& 2\ar[dl]&&4\ar[dr]\\
1 &&&& 5 &         &1 &&&&5}
$$
\end{example}

\vspace{5mm}

By the construction of rigid good modules, we have the following lemma.

\begin{lemma}\cite{JS, JSY}
The indecomposable summands of a rigid good $D$-modules 
$X(d)$ that are supported at a vertex $u$ are totally ordered by $\leq_u$.
\end{lemma}

{\color{black}
\subsection{Construction of Richardson elements}\label{REconstruction} Observe that a standard parabolic algebra in $\mathfrak{gl}_n$ can be naturally identified with the endomorphism algebra of a projective representation of a linear quiver. For instance, $\mathfrak{p}^+_U\leq \mathfrak{gl}_5$ with 
$U=\Pi\minus \{\alpha_2\}$ can viewed as $\mathrm{End}(P_1^2\oplus P_2^3)$ for projective representation $P_1^2\oplus P_2^3$ of the quiver $$\xymatrix{1 \ar@{->} [r]&2},$$ where the number of vertices  is the number of Levi-blocks of $\mathfrak{p}^+_U$ and the multiplicities $2$ and $3$ of $P_1$ and $P_2$ are the sizes of the Levi-blocks. This identification is due to the following,
$$\hom(P_1, P_1)=\hom(P_2, P_2)= \hom(P_2, P_1)=k \text{ and } \hom(P_1, P_2)=0.$$
Similarly, a seaweed Lie algebra $\mathfrak{q}_{S, T}$ can be view as an endomorphism algebra of a projective module of a quiver of 
type $\mathbb{A}$. In both cases, the nilpotent radical can then be identified with the Jacobson radical of the endomorphism algebra. 
\begin{example} \label{newexample}
Consider $\mathfrak{q}_{S, T}\leq \mathfrak{gl}_8=\mathfrak{gl}(V)$ with $S=\{\alpha_1, \alpha_4, \alpha_6, \alpha_7\}$ and $T=\Pi\minus \{\alpha_5, \alpha_7\}$. 
$$\mathfrak{q}_{S, T}=
\left(
\begin{matrix}
 $*$& $*$&&&& &&\\  
 $*$& $*$&&&&&&
\\  $*$& $*$& $*$&&& &&
\\  $*$& $*$& $*$& $*$& $*$ &  $*$& $*$& $*$
\\  $*$& $*$& $*$& $*$& $*$ &  $*$& $*$& $*$\\
 &&&{}& & $*$& $*$& $*$ \\
&&&{ }& & &$*$&$*$ \\
&&&& & &$*$&$*$\\
\end{matrix}
\right).
$$

Let $V=\oplus_iV_i$ with $V_i$s determined by the Levi-blocks and of dimension $2, \;1, \;2, \;1, \;2$, respectively. We have the following embeddings: 

$$
\xymatrix{V_1\oplus V_2\oplus V_3& \ar@{^{(}->}[l] V_2\oplus V_3  \ar@{^{(}->}[l] & \ar@{^{(}->}[l]   V_3  \ar@{^{(}->}[r] & V_3\oplus V_4 \ar@{^{(}->}[r] & 
V_3\oplus V_4\oplus V_5}
$$
This is the projective representation $P=P_1^1\oplus P_2\oplus P_3^2\oplus P_4\oplus P_5^2$ of the quiver $Q$
$$ \xymatrix{1& \ar@{->}[l] 2  & \ar@{->}[l]   3  \ar@{->}[r] & 4 \ar@{->}[r] & 5}$$
Denote by $W_i$ the subspace at step $i$ (starting from the left) in the embedding sequence. 
Restricting a point in the nilpotent radical $\mathfrak{n}_{S, T}$ to the subspaces $W_i$ produces a map of the neighboring spaces (in the direction opposite to the inclusions) and the relations of generic maps are exactly the defining relations of $D$ associated to $Q$. 
Let ${Q}_{S, T}$ be the Lie group associated to $\mathfrak{q}_{S, T}$,   $c_i=\dim W_i$ and $c=(c_i)$. Let $\mathrm{Rep}(D, P)$ be the  variety  of good $D$-modules with dimension vector $c$ 
as below
$$\{M\in \Pi_{i\rightarrow j\in \tilde{Q}_1}\hom(k^{c_i}, k^{c_j})\mid M \text{ satisfies the relations }
\mathcal{I} \text{ and } M|_Q=P\},$$
where $\mathcal{I}$ is as defined in Section \ref{introducingD}. We have 
\begin{itemize}
\item[(1)] $\mathfrak{n}_{S, T}=\mathrm{rad}\mathrm{End} P$ and ${Q}_{S, T}=\mathrm{Aut}P$, the automorphism group of $P$.

\item[(2)]  The three adjoint actions 
 ${Q}_{S, T}$ on  $\mathfrak{n}_{S, T}$, $\mathrm{Aut}P$ on $\mathrm{rad}\mathrm{End} P$ and 
$\mathrm{Aut}P$ on $\mathrm{Rep}(D, P)$
are essentially the same. In particular, open orbits correspond to open orbits. 

\item[(3)] By Voigt's Lemma \cite{Voigt}, a rigid $D$-module in $\mathrm{Rep}(D, P)$ implies an open $\mathrm{Aut}P$-orbit in 
$\mathrm{Rep}(D, P)$. This then implies the existence of Richardson elements in $\mathfrak{n}_{S, T}$ and we can construct 
a Richardson element from the rigid $D$-module (see what follows.)
\end{itemize}
\end{example}
}

We now decribe how to construct a Richardson element $r(d)$ for $\mathfrak{q}_{S, T}\subseteq \mathfrak{gl}_n$ 
from the rigid module $X(d)$. Note 
that $X(d)$ is constructed based on data contained in $\mathfrak{q}_{S, T}$.
Let $$X(d)=\bigoplus_{i} X^i$$ be a decomposition of $X(d)$ into indecomposable summands and let $n=\sum_id_i$. 
For each summand $X^i$ that is $\Delta$-supported at $j$, choose a number $x_{ij}$, where $$\sum_{l<j} d_l < x_{ij} \leq\sum_{l\leq j} d_l$$
such that $x_{ij}\neq x_{lj}$ for two different summands $X^i$ and $X^l$.  
If $X^i$ is $\Delta$-supported at both $s$ and $t$ with $s<t$, but not at $s+1,\cdots,t-1$, then  the matrix $r(d)\in \mathfrak{gl}_n$ has a $1$ at 
either $(x_{is},x_{it})$ or $(x_{it},x_{is})$, depending on which root-space belongs to $\mathfrak{q}_{S,T}=\mathrm{End}_A(P(d))$. All
other entries in $r(d)$ are equal to $0$.
Note that
$$\mathrm{End}_D(X(d))\cong \mathrm{stab}_{\mathfrak{q}_{S.T}}(r(d)).$$ 

The matrix $r(d)$ is also
the adjacency matrix of an oriented graph with components corresponding to indecomposable summands of $X(d)$. 
See  the example below for an illustration and \cite{baur, JSY, BHRR} for more detail.

\begin{example} \label{exampled}
The rigid modules $X(d)$ in Example \ref{example2} and \ref{example3} correspond, 
respectively,  to the parabolic subalgebra of $\mathfrak{gl}_5$ with Richardson element $r_1$ and the seaweed subalgebra of $\mathfrak{gl}_8$, which is exactly the seaweed given in Example \ref{newexample}, 
with Richardson element $r_2$ as follows 
$$r_1=
\left(
\begin{matrix}
{0}&{ 0}&{ 1}&{ 0}&{ 0}\\  {0}&{ 0}&{0}&{0}&{ 1}
\\ &&{ 0}&{ 1}&{ 0}
\\&&&{ 0}&{ 0}
\\ &&&{ 0}&{ 0}
\end{matrix}
\right) \;\mbox{ and }
r_2=
\left(
\begin{matrix}
{0}&{0}&&&& &&\\  
{0}&{0}&&&&&&
\\ {0}&{1}&{0}&&& &&
\\ {1}&{0}&{0}&{0}&{0} & {1}&{0}&{0}
\\ {0}&{0}&{1}&{0}&{0} & {0}&{0}&{1}\\
 &&&{}& & {0}&{1}&{ 0} \\
&&&{ }& & &{0}&\ {0} \\
&&&& & & {0}& {0}\\
\end{matrix}
\right).
$$
The two subalgebras are direct sum of the Cartan subalgebras and root spaces
in the bold faces, respectively. The corresponding oriented graphs are.
$$\xymatrix{P_1 & P_2 \ar[l] & P_3 \ar[l], && P_1 \ar[rr] & & P_3 & \ar[l] P_4 & P_5 \ar[l] \\ P_1 && P_3 \ar[ll] && P_1 \ar[r] & P_2 \ar[r] & P_3 & & P_5 \ar[ll]}$$
\end{example}

\section{Stabilisers of Richardson elements in type $\mathbb{A}$}

Consider the quiver $Q$ of type $\mathbb{A}_m$ of arbitrary orientation as in Section 3. 
For $D$-modules $M$ and $N$, let $$\Hom_D(N,M)_i=\{f_i\mid f \in \Hom_D(N,M)\}$$
be the space of homomorphisms from $N$ to $M$ restricted to vertex $i$. Let 
$\mathrm{End}_D(M)_i$ and $\mathrm{Aut}_D(M)_i$ be defined similarly.
We study the structure of the endomorphism algebra of a good rigid $D$-module 
and its restriction to a vertex.  

Let $$X(d)=\bigoplus_{i}(X^i)^{n_i}$$ be a good rigid $D$-module with $_{A}X(d)=P(d).$ We
order the summands such that 
$X^i<_mX^{i+1}$ for all indecomposable summands $X^i$.

\subsection{Restriction to $m$, when $m$ is a source}

Let $V=X(d)_m$. Then $$\mathrm{End}_D(X(d))_m\subseteq \mathfrak{gl}(V).$$
\begin{lemma} \label{seaweedpara}
The subalgebra $\mathrm{End}_D(X(d))_m\subseteq \mathfrak{gl}(V)$ is parabolic. 
\end{lemma}
\begin{proof}
By the construction of $X(d)$, $\mathrm{Hom}_D(X^i,X^j)_m$ is one dimensional for $i\leq j$ and zero otherwise. Further
we may choose basis elements $r_{ji}\in \mathrm{Hom}_D(X^i,X^j)_m$ for all $i$ and $j$, such that 
$r_{ji}r_{kl}=r_{jl}$ if $i=k$ and zero otherwise. The lemma follows.
\end{proof}

\subsection{Restriction to $m$, when $m$ is a sink.}
Recall that $\alpha_{m-1}:m-1\rightarrow m$ is the arrow that ends at $m$. 
For $D$-modules $M$ and $N$, let 
$$\mathrm{Hom}_D(M,N)^0_m=\{\overline{f}_m:M_m/\alpha_{m-1}(M_{m-1})\rightarrow N_m/\alpha_{m-1}(N_{m-1})|f\in\mathrm{Hom}_D(M,N)\}.$$
Let $$V=X(d)_m/\alpha_{m-1}(X(d)_{m-1}).$$ Then 
$\mathrm{End}_D(X(d))^0_m\subseteq \mathfrak{gl}(V)$. 

\begin{lemma} \label{parapara}
The subalgebra $\mathrm{End}_D(X(d))^0_m\subseteq \mathfrak{gl}(V)$ is parabolic.
\end{lemma}
\begin{proof}
The proof is similar to the proof of Lemma \ref{seaweedpara}.
\end{proof}

We remark that $\mathrm{End}_D(X(d))^0_m$ is in fact
isomorphic to $\mathrm{End}_A(P_m^{d_m})$.
\subsection{Stabilisers of indecomposable rigid modules}

There is an obvious 
embedding of endomorphism rings 
$$\prod_{i}\mathrm{End}_D(X^i)^{n_i}\subseteq \mathrm{End}_D(X(d)).$$
As before let $1=i_0<i_1< \cdots< i_{t+1}=m$ be the admissible vertices
of $Q$. 

Let $v$ be a sink in $Q$ and let $\alpha$ be an arrow ending at $v$. Then
$\alpha\alpha^*$ induces a nilpotent endomorphism $x$ of $X^i$.
We
have such an endomorphism 
$x_s:X^i\rightarrow X^i$ for each interval $[i_{s-1},i_s]$ with $s=1,\cdots,t+1$,
where $x_s$ can be zero, depending on the intersection of $\mathrm{supp}_{\Delta}(X^i)$ and $[i_{s-1}, i_s]$. 
Note that $x_s$ is zero on  vertices that are not in the interval $[i_{s-1},i_s]$.
Let $m_s\geq 1$ be the smallest integer
such that ${x_s}^{m_s}=0$. In fact, $$m_s= \mid\mathrm{supp}_{\Delta}(X^i)\cap [i_{s-1}, i_s]\mid.$$

\begin{lemma} \label{prettylemma}
The map $y_s\mapsto x_s$ induces an isomorphism $$\frac{k[y_1,\cdots,y_{t+1}]}{<y_sy_l=0 \mbox{ for } s\neq l, y_s^{m_s}=0>}
\stackrel{\simeq}{\longrightarrow} \mathrm{End}_D(X^i)$$
\end{lemma}
\begin{proof}
Let $x_s$ and $x_l$ be two arbitrary endomorphisms, induced by $\alpha\alpha^*$ and $\beta\beta^*$, respectively,
where $\alpha$ and $\beta$ are two different arrows. 
Clearly $x_sx_l=0$ if $\alpha$ and $\beta$ end at different sinks. Otherwise,  
 $x_sx_l=0$ follows from the relation $\alpha^*\beta=0$. 
By definition, $x_s^{m_s}=0$. So the map is well-defined. 

By the construction of $X^i$, $\mathrm{End}_D(X^i)$ is generated by the $x_s$ and
thus the map is surjective.
The intersection of the images of $x_s$ and $x_l$ is zero if $s\neq l$, and so the 
injectivity follows. 
\end{proof}

By the embedding $$\mathrm{End}_D(X(d))\subseteq \mathrm{End}_A(P(d))\subseteq \mathfrak{gl}_n$$
and the construction of Richardson elements discussed in Section \ref{REconstruction}, each element of $\prod_{i}\mathrm{End}_D(X^i)^{n_i}$ can be explicitly described in
terms of matrices. They can also be described in terms of the oriented graphs
constructed from $X^i$. 

\begin{example}
We use the two Richardson elements from Example \ref{exampled}. In both cases, there are 
two indecomposable summands $X^1$ and $X^2$ in $X(d)$. 
We use two different colours to describe 
$\mathrm{End}_D(X^1)$ and $\mathrm{End}_D(X^2)$ in both cases,
$$\mathrm{End}X^1\oplus \mathrm{End}X^2(\subseteq \mathrm{stab}(r_1)):
\left(
\begin{matrix}
{ a}&{ 0}&{ b}&{ c}&{ 0}
\\ { 0}&{ \color{blue} d}&{ 0}&{ 0}&{ \color{blue} e}
\\ &&{ a}&{ b}&{ 0}
\\ &&&{ a}&{ 0}
\\ &&&{ 0}&{ \color{blue} d}
\end{matrix}
\right) $$
and 
$$\mathrm{End}X^1\oplus \mathrm{End}X^2(\subseteq \mathrm{stab}(r_2)):
\left(
\begin{matrix}
{ a}&{ 0}&&&&&&\\  
{ 0}&{  \color{blue}  e}&&&&&&
\\ { 0}&{  \color{blue} f}&{  \color{blue} e}&&& &&
\\ { b}&{ 0}&{ 0}&{ a}&{ 0} & { c}&{ d}&{ 0}
\\ { 0}&{  \color{blue}  g}&{  \color{blue}  f}&{ 0}&{  \color{blue}  e} & { 0}&{ 0}&{  \color{blue}  h}\\
 &&&{}& & {a}&{ c}&{ 0} \\
&&&{ }& & &{a}&{0} \\
&&&& & &{0}&{ \color{blue}  e}\\
\end{matrix}
\right),
$$
where for instance in the first matrix,  the ${  a}$-entries are $a\cdot 1_{X^1}$
the ${  b}$-entries are $b\cdot x^1_1$ with $x^1_1: X^1\rightarrow X^1$, the ${  c}$-entries are $c\cdot (x^1_1)^2$
and the $e$-entry is $e\cdot x^2_1$ with $x^2_1: X^2\rightarrow X^2$. 

The non-zero off-diagonal entries in the matrices correspond to non-trivial paths in the oriented graphs as follows.
$$\xymatrix{& P_1 & P_2 \ar[l]^{ b} &P_3 \ar[l]^{ b} \ar@/_/[ll]_{ c} &&  P_1 
\ar[rr]_{ b} & & P_3 & \ar[l]^{ c} P_4 &P_5\ar[l]^{ c} \ar@/_/[ll]_{ d} \\ & P_1 && P_3 
\ar[ll]^{\color{blue} e} &&   P_1 \ar@/^/[rr]^{\color{blue} g} \ar[r]_{\color{blue} f} & P_2 
\ar[r]_{\color{blue} f} & P_3 & & P_5. \ar[ll]_{\color{blue} h}}$$
\end{example}

\section{The main result}

\begin{theorem} \label{maintheorem}
Let $\g$ be a simple Lie algebra of type $\mathbb{B}$, $\mathbb{C}$ or $\mathbb{D}$. Then any seaweed 
in $\g$ has a Richardson element. 
\end{theorem}

{\color{black}We continue to use the notation from Section 2. The proof of the theorem
is split into two  cases. The first one (see Theorem \ref{lemmam1})
deals with   the general situation, where we assume $(i)$ and $(ii)$ in Section \ref{generalassumption}. 
We will use quiver representations
and results from Section 4 to verify that the sufficient condition in Lemma \ref{keylemma} holds,
and so Richardson elements exist.
The second one deals with the special situation, where $\g$ is of type $\mathbb{D}$ and $(\epsilon,\omega)=(2,1)$.
Unlike the general case,  there can be root spaces that are not contained in $\a_{S,T}+\c_{S,T}$ 
(cf Lemma \ref{decomposition}). That is, $\mathfrak{q}_{S, T}$ is not
necessarily equal to 
$\a_{S,T}+\c_{S,T}$. }

\begin{theorem} \label{lemmam1}
Let $\mathfrak{q}_{S,T}\subseteq\g$ be a seaweed, where $\g$ is of type $\mathbb{B}$ or 
$\mathbb{C}$, or of type $\mathbb{D}$ with $(\epsilon,\omega)\neq (2,1)$. 
Then $\mathfrak{q}_{S,T}$ has a Richardson element. 
\end{theorem}
\begin{proof}
 Note that we only need to consider the situation, where $\epsilon>\omega$ and neither 
$S$ nor $T$ is equal to $\Pi$ or $\emptyset$. We may also assume $(ii)$ in Section \ref{generalassumption}.

Let $P=P(d)$ be a projective $A$-module such that $$\mathrm{End}_A(P(d))\simeq \a_{S,T},$$
where $P(d)$ is a projective representation of a quiver $Q$ of type $\mathbb{A}_m$ and the labeling of the vertices and arrows of 
$Q$ is as in Section \ref{introducingD}. Let 
$r\in \mathfrak{n}_\a$ be a Richardson element constructed from the good rigid $D$-module $X(d)$ as in Section \ref{REconstruction}.
Fix an embedding
$$\mathrm{End}_A(P(d))\subseteq \g$$ such that $$\mathrm{End}_A(P(d))=\a_{S,T} \text{ and }
\mathrm{End}_A(P(d))_m=\mathfrak{l},$$ where $\mathfrak{l}=\a_{S,T}\cap \c_{S,T}$ is as in Section 2.2. 
By the condition $\epsilon>\omega$ on $\a_{S,T}$, the vertex $m$ is a source. So 
$$\mathrm{stab}_{\a_{S,T}}(r)_{|\mathfrak{l}}=\mathrm{End}_D(X(d))_m$$ is a 
parabolic in $\mathfrak{l}$, by Lemma \ref{seaweedpara}. 

We order the summands $X(d)$ from big to small with respect to the order
$\leq_m$, so that $\mathrm{End}_D(X(d))_m$ is standard upper triangular.
Suppose that the sizes of the blocks in the Levi-subalgebra of $\mathrm{End}_D(X(d))_m$ are $c_1, \dots, c_l$
and let $$\hat{c}=(c_l, c_{l-1}+c_l, \dots,  \sum_{j\geq i}c_j, \dots, \sum_{j\geq 1}c_j).$$
Let $B$ be the path algebra of  the linearly oriented quiver $\mathbb{A}_l$ with the
unique sink $l$ and let 
$$P(\hat{c})=P_1^{\hat{c}_1}\oplus \dots \oplus P_l^{\hat{c}_l},$$ 
a projective representation of this linear quiver.
Denote by $F$ the algebra of the associated double quiver with relations, 
defined in the same way as the algebra $D$ in Section \ref{introducingD}, and let $X(\hat{c})$ be the rigid good $F$-module.

Note that 
$$
\mathrm{End}_B(P(\hat{c}))=\mathrm{End}_D(P(\hat{c}))
$$
and  by abuse of notation we  let 
$$
\mathrm{End}_B(P(\hat{c}))_l^0=\mathrm{End}_D(P(\hat{c}))_l^0,
$$
which is in fact isomorphic to $\mathrm{End}_B(P_l^{\hat{c}_l})$.
Choose $\g'$ and $U$ (see Section 2.3), and an embedding $$\mathrm{End}_B(P(\hat{c}))\subseteq \g'$$ such that
$$\mathrm{End}_B(P(\hat{c}))=\a_U \mbox{ and }\mathrm{End}_B(P(\hat{c}))^0_l=\mathfrak{l}.$$ 
Let $r'\in \a_{U}$ be a Richardson element corresponding to $X(\hat{c})$, where the summands of $X(\hat{c})$ are
ordered from small to big with respect to the order $\leq_l$, so that $$\mathrm{End}_F(X(\hat{c}))^0_l=\mathrm{stab}_{\a_U}(r')_{|\mathfrak{l}}$$
is standard upper triangular. Both $\mathrm{End}_F(X(\hat{c}))^0_l\subseteq \mathfrak{l}$ and $\mathrm{End}_D(X(d))^0_m\subseteq \mathfrak{l}$ 
are standard upper triangular with Levi blocks of equal sizes, and so
$$\mathrm{stab}_{\a_{S,T}}(r)_{|\mathfrak l}=\mathrm{stab}_{\a_U}(r')_{|\mathfrak{l}}.$$ 
Then $\mathfrak{q}_{S,T}$ has a Richardson element by Lemma \ref{keylemma}.
\end{proof}

The following example illustrate the construction of $X(\hat{c})$ in the proof above. 

\begin{example}
Let $Q$ be the quiver 
$$\xymatrix@=5mm{ 1\ar[r]& 2\ar[r]&3&4\ar[l],}$$
where $m=4$ is a source.
Let $d=( 3, 1, 3, 4)$. Then the rigid module $X(d)$ is 
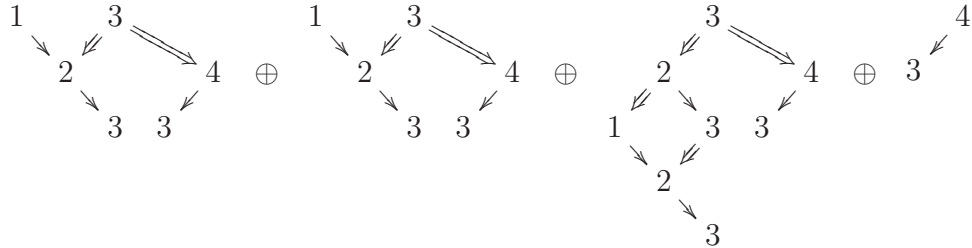
\begin{center}
$\xymatrix@=2mm{  
1\ar [dr]&&3\ar@{=>}[dl] \ar@{=>}[drr]&& & \\
& 2\ar[dr]&&&4\ar[dl]& \oplus\\
&&3&3&&}$ 
$\xymatrix@=2mm{  
1\ar [dr]&&3\ar@{=>}[dl] \ar@{=>}[drr]&& & \\
& 2\ar[dr]&&&4\ar[dl]& \oplus \\
&&3&3&&}$ 
$\xymatrix@=2mm{  
&& 3\ar@{=>}[drr] \ar@{=>}[dl] & \\
& 2\ar@{=>}[dl]\ar[dr]    &&&4\ar[dl]  & \oplus \\
1\ar [dr]&&3\ar@{=>}[dl]& 3 & \\
& 2\ar[dr]\\
&&3&}$ 
$\xymatrix@=2mm{  
&4\ar[dl]& \\
3  &}$ 
\end{center}
where the summands are ordered from big to small with respect to $\leq_4$. In the base interval 
$[3, 4]$, 
the last summand has the smallest $\Delta$-support $\{4\}$ and so is the smallest one, the other
summands have the same   $\Delta$-support $\{3, 4\}$ and so 
they are compared at next interval $[1, 3]$, in which case representations with smaller supports are actually bigger. We have 
\begin{itemize}\item[(1)]
$(\mathrm{End}X(d))_4=\begin{pmatrix} * &* &*&* \\ * &* &*&*\\ & &*&*\\ & & &*\end{pmatrix}$, 
a standard upper triangular parabolic in $\mathfrak{gl}_4$.

\item[(2)] The number of blocks $l=3$ with the sizes 2, 1, 1  and $\hat{c}=(1, 2, 4)$.

\item[(3)] The representation $P(\hat{c})=P_1 \oplus P^2_2\oplus P_3^4$ is the projective representation of the quiver $\xymatrix@=5mm{ 1\ar[r]& 2\ar[r]&3}$.

\item[(4)] The rigid $F$-module $X(\hat{c})$ is as follows. 
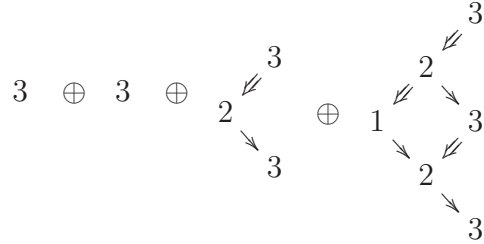
\begin{center} 
$\xymatrix@=2mm{&\\& \\ 3 &\oplus }$
 $\xymatrix@=2mm{&\\& \\ 3 &\oplus }$
 $\xymatrix@=2mm{  
\\
&3\ar@{=>}[dl]&     \\
 2\ar[dr] &&\oplus\\
&3&}$
 $\xymatrix@=2mm{  
&& 3 \ar@{=>}[dl] & \\
& 2\ar@{=>}[dl]\ar[dr]      &\\
1\ar [dr]&&3\ar@{=>}[dl]& \\
& 2\ar[dr]\\
&&3&}$ \\
\end{center}
The summands  are ordered from small to big with respect to $\leq_3$. The space 
$\mathrm{End}(X(\hat{c}))^0_3$ is the space of induced homomorphisms between the top 3's in the summands and indeed

$$\mathrm{End}(X(\hat{c}))^0_3=(\mathrm{End}X(d))_4.$$
 
\end{itemize}
\end{example}

Denote by $E_{ij}$ the  elementary matrix with $1$ at $(i, j)$-entry and $0$ elsewhere. 

\begin{theorem} \label{lemmam2}
If $\g$ has type $\mathbb{D}$ and
$(\epsilon,\omega)=(2,1)$, then the seaweed $\mathfrak{q}_{S,T}\subseteq \g$ has a Richardson element.
\end{theorem}
\begin{proof}  
Let $W=S\backslash \{1\}$ and 
$$
\mathfrak{q}_{S, T}^1=\bigoplus_{\begin{tabular}{c} $\alpha$  supported at \\ $\alpha_1$, $\alpha\in \Phi_S^+$\end{tabular}}\g_{\alpha}.
$$
Then $$\mathfrak{q}_{S, T}=\mathfrak{q}_{W, T}\oplus \mathfrak{q}_{S, T}^1 \mbox{ and }
\mathfrak{n}_{S, T}=\mathfrak{n}_{W, T}\oplus \mathfrak{q}_{S, T}^1. 
$$
Note that $\mathfrak{q}_{W, T}$ is a seaweed of type $\mathbb{A}$ and so 
it has a Richardson element $r\in \mathfrak{n}_{W, T}$.

Since $\alpha_1\not\in T$, we 
have $[\mathfrak{q}_{S,T},\mathfrak{q}_{S, T}^1]
\subseteq \mathfrak{q}_{S, T}^1$. Then $r+r'$ with $r'\in\mathfrak{q}_{S,T}^1$ is a Richardson element
in $\mathfrak{q}_{S, T}$ if 
$[\mathrm{stab}_{\mathfrak{q}_{W,T}}(r), r']=\mathfrak{q}_{S, T}^1$. Indeed, let $x+x'\in\mathfrak{n}_{W,T}\oplus \mathfrak{q}_{S,T}^1$
and let $y\in \mathfrak{q}_{W,T}$ and $y'\in \mathrm{stab}_{\mathfrak{q}_{W,T}}(r)$ such that $$[y,r]=x\mbox{ and }[y',r']=x'-[y,r'].$$
Then $$[y+y',r+r']=x+x'$$ and so $r+r'$ is a Richardson element.
Hence, to prove the lemma, it suffices to show $$[\mathrm{stab}_{\mathfrak{q}_{W,T}}(r), r']
=\mathfrak{q}_{S, T}^1.$$

We choose the representation of $\g$ given by $2n\times 2n$-matrices anti-symmetric to the 
anti-diagonal, $$\g_{\alpha_1}=k\cdot (E_{n-1,n+1}-E_{n,n+2})\mbox{ and  } 
\g_{\alpha_2}=k\cdot (E_{n-1,n}-E_{n+1,n+3}).$$ Then $\mathfrak{q}_{S, T}^1$
has a basis $E_{l,n+1}-E_{n,n-l+1}$ for $l=n-a+2,\cdots,n-1$, where 
$$ a=\left \{ \begin{tabular}{ll}
min\{$p|p>\epsilon, \alpha_p\not\in S$ \} & if such a $p$ exists, \\ 
$n+1$ & otherwise.
\end{tabular} \right.$$ Fix an embedding 
$\mathfrak{gl}_n\subseteq \g$: $$E_{ij}\mapsto E_{ij}-E_{2n-j+1,2n-i+1}.$$
We may assume the Richardson element $r$ is $r(d)$, constructed from a rigid good module $X(d)=\oplus_iX^i$, 
as in Section \ref{REconstruction}. Recall that the underlying quiver $Q$ is of type $\mathbb{A}_m$, where 
$m$ is the number of Levi-blocks in $\mathfrak{q}_{W,T}$.

Choose an embedding 
$$\mathrm{End}_A(P(d))\subseteq \g$$ 
such that 
$$\mathrm{End}_A(P(d))=\mathfrak{q}_{W,T}
\text{ and }\mathrm{End}_D(X(d))=\mathrm{stab}_{\mathfrak{q}_{W,T}}(r).
$$
Then $$\mathrm{End}_D(X^i)\subseteq \mathrm{stab}_{\mathfrak{q}_{W,T}}(r),$$
{with $$1_{X^i}=\sum_{j\in\mathrm{Supp}_\Delta(X^i)}(E_{x_{ij},x_{ij}}-E_{2n-x_{ij}+1,2n-x_{ij}+1}),$$ where the $x_{ij}$ are constructed from the summand $X^i$ as in Section \ref{REconstruction}.}

As $\epsilon>\omega$, the vertex $m$ is a source. The $m$th Levi-block in $\mathfrak{q}_{W,T}$ is of rank  $1$, i.e. $d_m=1$, as $\alpha_2 \not\in W$, and it is spanned by $E_{n,n}-E_{n+1, n+1}$. 

We will decompose $\mathfrak{q}^1_{S,T}$ as direct sum of subspaces $V_i$, determined by the summands of $X(d)$, and then construction a 'Richardson element' $r_i$ in each subspace, in the sense that 
$$[\frak{q}_{S, T}, r_i]=V_i.$$
Observe that for any $x\in \mathfrak{g}_{\alpha}\subseteq \mathfrak{q}^1_{S,T}$, 
$$[E_{n, n}-E_{n+1, n+1}, x]=x$$ 
and when $i <n$,
$$[E_{i, i}-E_{2n-i+1, 2n-i+1}, x]=\left\{ \begin{tabular}{ll} $x$ & if $\alpha$ is supported at $\alpha_{n-i+1}$;\\ 0 & otherwise.
\end{tabular}\right.$$

The orientation of the arrows at vertex $m-1$ determines the construction of $V_i$.
There are two cases to be considered. 

Note that  there is a unique summand of $X(d)$, say $X^{i_0}$, such that $X^{i_0}|_Q$ contains $P_m$
as a summand and $X^{i_0}\not= P_m$, i.e. $X^{i_0}|_Q$ has at least two summands and so
$$
1_{X^{i_0}}-(E_{n, n}-E_{n+1, n+1})\not= 0. 
$$ 

Case (1) The vertex $m-1$ is non-admissible.
In this case, 
$$
\dim \mathfrak{q}^1_{S,T}=d_{m-1}.
$$

Note that each summand $P_{m-1}$ is contained in  a different indecomposable summand of $X(d)$ and we
 may assume these summands are $X^1, \dots, X^{d_{m-1}}$ and $x_{i(m-1)}=\sum_{s\leq m-1}d_s-i+1$.
For $1\leq i\leq d_{m-1}$, let 
$V_i$ be the unique roots space $\frak{g}_{\alpha}$ such that 
\begin{itemize}\item[(i)]  when  $i\not= i_0$, $[1_{X^i}, \mathfrak{g}_{\alpha}]\not= 0$;
\item[(ii)]  when  $i= i_0$, $[1_{X^{i}}-(E_{n, n}-E_{n+1, n+1}), \mathfrak{g}_{\alpha}]\not= 0$. 
\\Note that  in this case, $x_{im}=n=\sum_{s}d_s$.
\end{itemize}
Then $$\frak{q}_{S, T}^1=\oplus_i V^i. $$

Let $r_j\in V_j$ and let $\delta_{ij}$ be Kronecker numbers. We have 
$$
[1_{X_i},r_j]=\left\{
\begin{tabular}{ll}$\delta_{ij}r_j$ & if $i\not= i_0$; \\
$2r_{i_0}$ & if $i=j=i_0$;\\
$r_j$ & if $i=i_0$ and $j\not= i_0$.
\end{tabular}
\right.
$$

Choose a nonzero element $r_i\in V_i$ and let $$r'=\sum_i r_i.$$

Case (2) The vertex $m-1$ is admissible. Then it is a sink and  

$$
\dim \mathfrak{q}^1_{S,T}>d_{m-1}.
$$

Let 
$$
V_i=
\left\{ \begin{tabular}{ll}
$\mathfrak{q}^1_{S,T}\cap \bigoplus_{[1_{X^i},\;\g_\alpha]\neq 0} \g_\alpha$ & if $i\not= i_0$;\\
$\mathfrak{q}^1_{S,T}\cap \bigoplus_{[1_{X^i}-(E_{n, n}-E_{n+1, n+1}),\;\g_\alpha]\neq 0} \g_\alpha$ & if $i= i_0$;
\end{tabular}
\right. $$
Note that $V_i$ can be $0$ and we have
$$V=\oplus_i V_i.$$
When $V_i\not= 0$, we let $\beta_i$ be the smallest root such that 
\begin{itemize}\item[(i)] when  $i\not= i_0$,
$[1_{X^i},\g_{\beta_i}]\neq 0$;
\item[(ii)] when  $i= i_0$, $[1_{X^i}-(E_{n, n}-E_{n+1, n+1}),\g_{\beta_i}]\neq 0$.
As in Case (1) (ii), $x_{im}=n$. 
\end{itemize}
Then for any non-zero $r_i\in \g_{\beta_i}$, 
$$V_i=\sum_j [k\cdot x^j,r_i],$$ 
where $x$ is the endomorphism induced 
by $\alpha\alpha^*$ with $\alpha$ the arrow from vertex $m-2$ to vertex $m-1$ (see Lemma \ref{prettylemma}). 
Choose a non-zero element $r_i\in \g_{\beta_i}$ and let $$r'=\sum_{i}r_i.$$

In both cases, 
$$[\mathrm{stab}_{\mathfrak{q}_{W,T}}(r),r']=\bigoplus_i V_i=\mathfrak{q}^1_{S,T}.$$ 
Therefore $\mathfrak{q}_{S,T}$ has a Richardson element.
\end{proof}

\begin{remark}
\begin{itemize}
\item[a)] The Richardson elements and their stabilisers can be explicitly constructed  
using results from \cite{JSY}, the proofs of Theorem \ref{lemmam1} and Theorem \ref{lemmam2}.
The work of Baur \cite{baur} on parabolic Lie algebras is also needed in the case Theorem \ref{lemmam1}. 
\item[b)] The method of Theorem \ref{lemmam2} can be generalised to Lie algebras of exceptional
types and therefore provide an explanation why Richardson elements do not exist for some seaweed Lie algebras of 
type $\mathbb{E}_8$.
\end{itemize}
\end{remark}

\newcommand{\q}{\mathfrak{q}}

We end this paper with an example of constructing Richardson elements, using the method discussed in Theorem \ref{lemmam2}.

\begin{example}
Let $\g=\mathfrak{so}_{10}$, a Lie algebra of type $\mathbb{D}_5$. 

\vspace{2mm}

(1)  Consider the seaweed Lie algebras 
$\q_{S, T}$ and $\q_{K, L}$ with $T=\{\alpha_2, \alpha_4, \alpha_5\}$, $S=\{\alpha_1, \alpha_3, \alpha_4\}$, $L=\{\alpha_2, \alpha_3, \alpha_4\}$ and 
$K=\{\alpha_1, \alpha_3, \alpha_5\}$. These two seaweeds are like those  discussed in the proof of Theorem \ref{lemmam2} and have 
the following shapes. The matrices are anti-symmetric to the anti-diagonal.

\vspace{5mm}

\begin{center}
$\left(\begin{tabular}{cccccccccc}
 *& & & & & & & & & \\
 *& *& *& *& &  *& & & &  \\
 *& *& *& *& &  *& & & &  \\
&&&  * &  &  * & &  & & \\
&&& * & * &  & *&  * & * &  \\
&&&&&     *& &  & &  \\
&&&&&   * & * & *  & *&  \\
&&&&&    & &  * & * &  \\
&&&&&    & &  * & * &  \\
&&&&&    && *& *& * \\
\end{tabular}\right)
$, \;
$\left(\begin{tabular}{llllllllllll}
 *&  *& & & & & & & & \\
& *& && & & & & &  \\
& *& *& *& &  *& & & &  \\
& *& *&  *&  & * & & & & \\
& *& *& * & * &  & * &  * & &  \\
&&&&&     *& &  & &  \\
&&&&&    *& * & *  & &  \\
&&&&&     *& *&  * &  &  \\
&&&&&     *& * &  * & * & *  \\
&&&&&    &&&& *\\
\end{tabular}\right)
$.
\end{center}

\vspace{5mm}

(2) The seaweeds $\q_{S\backslash \{1\}, T}$ and $\q_{K\backslash \{1\}, L}$ are of type 
$\mathbb{A}$. They are isomorphic to $\mathrm{End}_A(P(c))$ and $\mathrm{End}_A(P(d))$ 
of the following quivers, respectively, where $c=(1, 2, 1, 1)$ and $d=(1, 1, 2, 1)$, 
$$\xymatrix{1&2\ar[l]\ar[r]& 3&4\ar[l], & &1\ar[r]&2&\ar[l]3&\ar[l]4}.$$
In both case, the vertex $m$ in the proof of Theorem \ref{lemmam2} is $4$. The vertex $3$ $(=m-1)$ is a 
sink for  $\q_{S\backslash \{1\}, T}$ and is non-admissible for $\q_{K\backslash \{1\}, L}$.
We have
$$\q_{S, T}^1=V_1\oplus V_2 \;\;\mbox{ and } \;\;
\q_{K. L}^1=W_1\oplus W_2$$ with $$V_1=\g_{\alpha_1+\alpha_3+\alpha_4}, \;
V_2=\g_{\alpha_1}\oplus\g_{\alpha_1+\alpha_3}, \;W_1=\g_{\alpha_1+\alpha_3}\;\; \mbox{ and }\;\;
W_2=\g_{\alpha_1}.$$  

\vspace{5mm}

(3) The Richardson elements of $\q_{S\backslash \{1\}, T}$ and $\q_{K\backslash \{1\}, L}$  are 
as below. Entries of the same colour come from the same indecomposable summand. Note that for
$X(d)$, one of the indecomposable 
summands is a Verma module, so both the $4$th coloumn and row are zero. 
\vspace{5mm}

\begin{center}
$\left(\begin{tabular}{cccccccccc} 
 0& & & & &  & & & & \\
 1& 0& 0& 0& &   0&  & & &  \\
 0& 0& 0& \color{blue} 1&   &   0& & & &  \\
&&&  0&    &    0 & &  & &\\
&&& \color{blue} 1 & 0 &  &  0 &   0 &  0 &  \\
&&&&&     0& &  & &  \\
&&&&&    \color{blue} -1 & 0 & \color{blue} -1  & 0 &  \\
&&&&&    & &  0 & 0 &  \\
&&&&&    & &  0 & 0 &  \\
&&&&&    && 0&{-1}& 0 \\
\end{tabular}\right)
$, \;
$\left(\begin{tabular}{cccccccccc}  
 0&  1& & & & & & & & \\
& 0& && & & & & &  \\
& 1& 0& 0& &   0& & & &  \\
& 0& 0&  0&  &   0 & &  & & \\
& 0& 1& 0 & 0 &  &  0 &   0 & &  \\
&&&&&     0& &  & &  \\
&&&&&   0 & 0 & 0  & &  \\
&&&&&     -1&0 &  0 &  &  \\
&&&&&     0& 0 &  -1 & 0 & -1  \\
&&&&&    &&&& 0 \\
\end{tabular}\right)
$.
\end{center}

\vspace{5mm}

(4) The stabilizers of the two Richardson elements are as follows.

\vspace{5mm}
\begin{center}
$\left(\begin{tabular}{cccccccccc} 
   a& & & & & & & & & \\
  b&  a &   0&   \color{green} g& &   0& & & &  \\
  \color{red} f &   0&  \color{blue}c&   \color{blue}d& &   0& & & &  \\
&&&   \color{blue}c&  &   0 & & & & \\
&&&    \color{blue}e&  \color{blue}c &  &  0 &   0 &  0 &  \\
&&&&&      \color{blue}-c& &  & &  \\
&&&&&      \color{blue}-e&   \color{blue}-c&   \color{blue}-d &   \color{green}-g &  \\
&&&&&    & &    \color{blue}-c &   0 &  \\
&&&&&    & &   0 &   -a&  \\
&&&&&    &&  \color{red}-f&  -b&   -a\\
\end{tabular}\right)
$, \;
$\left(\begin{tabular}{cccccccccc} 
  a&   b& & & & & & && \\
&  a & & & & & & & &  \\
 &   c&  a&   0& &   0& & & &  \\
&  \color{red} f&  0&   \color{blue}e&  &   0 & &  & & \\
&  d&  c&    \color{green}g&  a  &  &  0 &   0 & &  \\
&&&&&       -a& &  & &  \\
&&&&&      \color{green}-g&   \color{blue}-e&   0 &   & \\
&&&&&      -c&  0 &    -a &  &  \\
&&&&&      -d&  \color{red}-f &   -c &   -a&  -b  \\
&&&&&    && &&   -a\\
\end{tabular}\right)
$.\\
\end{center}
There are two indecomposable direct summands in each of $X(c)$ and 
$X(d)$. The different colours indicate homomorphisms between different pairs of
summands. 

\vspace{5mm}

(5) Denote the two Richardson element in (3)
by $r_{S\backslash \{1\}, T}$ and $r_{K\backslash \{1\}, L}$. 
 The action of $\mathrm{stab}_{\q_{S\backslash\{1\}, T}}(r_{S\backslash \{1\}, T})$ on $\q^1_{S\backslash\{1\}, T}$
is equivalent to the natural action of
$\left(\begin{tabular}{ccc} 
a +{\color{blue} c}& 0&\color{green} g \\
  0&\color{blue}2c& \color{blue}d \\
 0&0& \color{blue}2c  \\
\end{tabular}\right)$
on 
$k^3$, although in the proof of Theorem \ref{lemmam2}, we only use the action of 
$\left(\begin{tabular}{ccc} 
a+ \color{blue}c & 0&0 \\
  0&\color{blue}2c& \color{blue}d \\
 0&0& \color{blue}2c  \\
\end{tabular}\right)$
with $r'=\left(\begin{tabular}{c} 
1\\ 0 \\  1
\end{tabular}\right)\in\mathfrak{q}_{S,T}^1$. The action of $\mathrm{stab}_{\q_{K\backslash\{1\}, L}}(r_{K\backslash \{1\}, L})$ on $\q^1_{K\backslash\{1\}, L}$
is equivalent to the natural action of $\left(\begin{tabular}{cc} 
a + \color{blue}e& 0 \\
  0&\color{blue}2e 
\end{tabular}\right)$
on   $k^2$, and $r'=\left(\begin{tabular}{c} 
1 \\   1
\end{tabular}\right)\in\mathfrak{q}_{K,L}^1$. So we have the following Richardson elements for the two seaweeds $\q_{S, T}$ and $\q_{K, L}$, with the red entries coming from the contributions of $r'$ in $\mathfrak{q}^{1}_{S, T}$ and $\mathfrak{q}^{1}_{K, L}$ respectively.

\vspace{5mm}
\begin{center}
$\left(\begin{tabular}{cccccccccc} 
 0& & & & &    & & & & \\
 1& 0& 0& 0&     &  \color{red} 1& & & &  \\
 0& 0& 0& \color{blue} 1&    &  0& & & & \\
&&&  0&  &     \color{red}1& &  & & \\
&&& \color{blue} 1 & 0   &  &  \color{red} -1 &  0 & \color{red}-1 &  \\
&&&&&     0&&   & &  \\
&&&&&    \color{blue} -1 & 0 & \color{blue} -1  & 0 &  \\
&&&&&    & &  0 & 0 & \\
&&&&&    & &  0 & 0 &  \\
&&&&&    && 0&{-1}& 0 \\
\end{tabular}\right)
$, \;
$\left(\begin{tabular}{cccccccccc}  
 0&  1&& & & & & & & \\
& 0& && & & & & &  \\
& 1& 0& 0& &  \color{red} 1& & & &  \\
& 0& 0&  0&  &  \color{red} 1 & &  & & \\
& 0& 1& 0 & 0 &  &  \color{red} -1&  \color{red} -1 & &  \\
&&&&&     0& & & &  \\
&&&&&   0 & 0 & 0  & &  \\
&&&&&     -1&0 &  0 & &  \\
&&&&&     0& 0 &  -1 & 0 & -1  \\
&&&&&    &&&& 0 \\
\end{tabular}\right)
$.
\end{center}
\end{example}


\vspace{5mm}

{\parindent=0cm 
BTJ:
Department of Mathematical Sciences,
NTNU in Gj\o vik, Norwegian University of Science and Technology,
2802 Gj\o vik, Norway. \\
Email: bernt.jensen@ntnu.no \\}

{\parindent=0cm
XS:
Department of Mathematical Sciences,
University of Bath,
Bath BA2 7JY,
United Kingdom.\\
Email: xs214@bath.ac.uk }
\end{document}